\documentclass[11pt]{article}

\usepackage{amsmath}
\usepackage{amsthm}
\usepackage{amssymb}
\usepackage{mathrsfs}
\usepackage{graphicx}
\def\div{{\,\rm div \,}}
\def\cc{{\cal C }}

\def\LL{{\,\rm L \,}}

\def\g{{\,\rm \gamma \,}}
\def\sign{{\,\rm sign \,}}

\def\T{{\cal T}}
\def\Id{{\,\rm Id \,}}

\def\CC{{\,\rm C\,}}
\def\WW{{\,\rm W\,}}
\def\qiq{{\quad\mbox{in}\quad}}

\def\o{{\,\rm o \,}}
\def\id{{\,\rm id \,}}

\def\sym{{\,\rm sym \,}}

\def\ii{{\,\rm i \,}}
\def\dist{{\,\rm dist \,}}

\def\SO{{\,\rm SO \,}}

\def\supp{{\,\rm supp \,}}

\def\B{{\,\cal B \,}}

\def\+M{{\,\rm M^{n\times n}_+ \,}}
\def\tr{{\,\rm tr \,}}

\def\qfq{{\quad\mbox{for}\quad}}

\def\ii{{\,\rm i \,}}

\def\supp{{\,\rm supp \,}}

\def\O{{\,\rm O\,}}

\def\lam{\lambda}

\def\t#1{{\tilde#1}}

\def\E{{\cal E}}

\def\U{{\cal U}}

\def\X{{\cal X}}
\def\Q{{\cal Q}}

\def\E{{\cal E}}
\def\L{{\cal L}}
\def\U{{\cal U}}
\def\W{{\cal W}}
\def\V{{\cal V}}
\def\F{{\cal F}}

\def\O{{\cal O}}
\def\A{{\cal A}}
\def\H{{\cal H}}
\def\G{{\cal G}}
\def\F{{\cal F}}
\def\Q{{\cal Q}}
\def\M{{\cal M}}

\def\K{{\cal K}}
\def\V{{\cal V}}
\def\N{{\cal N}}
\def\Z{{\cal Z}}

\def\RR{{\cal R}}

\newfont{\Blackboard}{msbm10 scaled 1200}

\newfont{\roma}{cmr10 scaled 1200}

\def\D{{\cal D}}

\def\<{{\langle}}
\def\>{{\rangle}}

\def\Ga{\Gamma}
\def\var{\varphi }
\def\si{\sigma}

\def\a{\alpha}
\def\b{\beta}

\def\Om{\Omega}

\newtheorem{thm}{{}\hskip\parindent Theorem}[section]
\newtheorem{lem}{{}\hskip\parindent Lemma}[section]
\newtheorem{pro}{{}\hskip\parindent Proposition}[section]

\newtheorem{cor}{{}\hskip\parindent Corollary}[section]

\newtheorem{rem}{{}\hskip\parindent Remark}[section]

\def\dfrac{\displaystyle\frac}

\def\pl{\partial}
\def\rw{\rightarrow}

\def\na{\nabla}

\def\be{\begin{equation}}
\def\ee{\end{equation}}
\def\beq{\arraycolsep=1.5pt\begin{eqnarray}}
\def\eeq{\end{eqnarray}}

\def\R{I\!\!R}

\def\n{\vec{n}}
\title{Strain Tensors and  Matching Property on Surfaces with the Gauss
curvature changing sign}
\date{}
\author{
Liang-Biao Chen and Peng-Fei Yao\\[0.2cm]
\nonumber
Key Laboratory of Systems and Control\\\nonumber
Institute of Systems Science,
Academy of Mathematics and Systems Science\\\nonumber
Chinese Academy of Sciences, Beijing 100190, P. R.
China\\\nonumber
School of Mathematical Sciences\\\nonumber
University of Chinese Academy of Sciences, Beijing 100049,
China\\\nonumber
e-mail: pfyao@iss.ac.cn}
\begin{document}
\maketitle
 \footnote{This work is supported by the National
Science Foundation of China, grants no. 12071463 and Key Research Program of Frontier Sciences,
CAS, no. QYZDJ-SSW-SYS011.}

\begin{quote}
\begin{small}
{\bf Abstract} \,\,\,We prove the regularity of solutions to the strain tensor equation on a region $S$ with the Gauss curvature changing sign.
Furthermore, we obtain the density property that smooth infinitesimal isometries are dense in the $W^{2,2}({S},\R^3)$ infinitesimal isometries. Finally, the matching property is established. Those results are  important tools in
obtaining recovery sequences ($\Ga$-lim sup inequality) for dimensionally-reduced shell theories
in elasticity.
\\[3mm]
\end{small}
\end{quote}

\section{Introduction and Main Results}
\def\theequation{1.\arabic{equation}}
\hskip\parindent
Let $M\subset\R^3$ be a surface with a normal $\n$ and let the middle surface of a shell be  an open set $S\subset M.$ Let $T^kS$ denote  all the $k$-order tensor fields on $S$ for an integer $k\geq0.$ Let $T^2_{\sym}S$ be all the $2$-order symmetrical tensor fields on $S.$ For $y\in \WW^{1,2}(S,\R^3),$ we decompose it into $y=W+w\n,$ where $w=\<y,\n\>$ and $W\in TS.$
For $U\in T^2_\sym S$ given, linear strain tensor of a displacement $y\in\WW^{1,2}({S},\R^3)$ of the middle surface ${S}$ takes the form
\be \sym DW+w\Pi=U\qfq x\in{S},\label{01}\ee where $D$ is the connection of the induced metric in $M,$  $2\sym DW=DW+D^TW,$ and $\Pi$ is the second fundamental form of $M.$
Equation (\ref{01}) plays a fundamental role in the theory of thin shells, see \cite{HoLePa, LeMoPa,LePa, LeMoPa1,Yao2017, Yao2018} and many others. When $U=0,$ a solution $y$ to (\ref{01}) is referred to as an {\it infinitesimal isometry.}

The type of equation (\ref{01}) depends on the sign of the curvature on the region S: It is elliptic if $S$ has positive curvature; it is parabolic if the curvature is zero but $\Pi\not=0$ on $S;$ it is hyperbolic if $S$ has negative curvature. When the curvature of the region $S$ changes its sign, problem (\ref{01}) is of changing type.

Here  we establish the regularity of solutions to (\ref{01}) when $S$ is a region with its curvature changing sign, which will be specified below.
Then it
is proved that smooth infinitesimal isometries are dense in the $W^{2,2}({S},\R^3)$ infinitesimal isometries on the region $S.$ Finally, the matching property is derived that
smooth enough infinitesimal isometries can be matched with higher
order infinitesimal isometries. Those results are  important tools in
obtaining recovery sequences ($\Ga$-lim sup inequality) for dimensionally-reduced shell theories
in elasticity, when the elastic energy density scales like $h^\b,$ $\b\in(2, 4),$ that is, intermediate
regime between "pure bending" ($\b=2$) and the von-K\'arm\'an regime ($\b=4$). Such results have been obtained for elliptic surfaces \cite{LeMoPa1}, developable surfaces \cite{HoLePa}, and  hyperbolic surfaces \cite{Yao2017, Yao2018}. Moreover, the case of degenerated hyperbolic surfaces has been studied in \cite{CY}. A survey on this topic is presented in \cite{LePa}.

Mixed type equations also arise naturally in many other areas.  A detailed account of the historical background and known results on mixed type equations and transonic flows is given in \cite{Mor}.
A detailed review on mixed type equations and Riemannian-Lorentzian metrics is presented in \cite{Otway}. The most intensively studied equation of mixed type is the Tricomi equation \cite{Trico}.

In this paper we study the mixed type equation (\ref{01}) which is very different from all the above cases. Equation (\ref{01}) is equivalent to a mixed type scalar equation of the form
\be\<D^2w,Q^*\Pi\>+\dfrac1\kappa\<Dw, X_0\>+\kappa(\tr_g\Pi)w=\kappa f+\dfrac1\kappa\<X_0,F\>+\<DF,Q^*\Pi\>\qfq x\in S,\label{Q1}\ee $\kappa\not=0,$ where $w\in\LL^2(S)$ is the unknown, $f\in\LL^2(S)$ and $F\in\LL^2(S,TS)$ are given,
and $\kappa$ is the Gauss curvature. The type of (\ref{Q1}) is subject to the sign of $\kappa.$ When $\kappa=0,$ there are two terms degenerating in (\ref{Q1}): The coefficient of $\<Dw-F,X_0\>$ becomes
from $+\infty$ to $-\infty$ as a point crosses a zero curvature curve from the positive curvature to the negative. The coefficient of a second derivative of $w$ along the direction of the zero principal curvature changes from positive to negative. Those situations challenge the analysis of (\ref{Q1}).

We observe that problem (\ref{Q1}) is equivalent to a vector field system as
\be\begin{cases}Dv=\nabla\n V+F\qfq x\in S,\\
\div_gV=-v\tr_g\nabla\n+f\qfq x\in S,
\end{cases}\label{Q2}\ee where $\nabla\n:$ $S_x\rw S_x$ is the second fundamental form operator, $(V,v)$ is the unknown where $V$ is a vector field and $v$ is function, and $(F,f)$ is as in (\ref{Q1}).
Problem (\ref{Q2}) changes  type when the curvature changes  sign. In order to have solutions to problem (\ref{01}), one normally starts to solve problem (\ref{Q1}) in the previous works as in \cite{CY,HoLePa,LeMoPa1,Yao2017}. However, it makes the solvability of (\ref{Q1}) extremely difficult that  the coefficient of $\<Dw-F,X_0\>$ varies between $+\infty$ and $-\infty$ when the curvature changes  sign. Instead of analysing (\ref{Q1}), here we begin from (\ref{Q2}) and then obtain solutions to (\ref{01}). The hard  task is to achieve regularity of solutions to (\ref{Q2}) by some priori estimates near the zero curvature curve. Thanks to the Bochner technique and the tensor analysis,  the regularity analysis has been complete in Section 3.

We state our main results as follows.
Let $S\subset M$ be given by
\be S=\{\,\a(t,s)\,|\,(t,s)\in[0,a)\times(-b_0,b_1)\,\},\quad a>0,\quad b_0>0,\quad b>0,\label{Q0}\ee where $\a:$ $[0,a)\times[0,b]\rw M$ is an imbedding map which is a family of regular curves with two parameters $t,$ $s$ such that
\be \Pi(\a_t(t,s),\a_t(t,s))>0,\quad\mbox{for all}\quad (t,s)\in[0,a)\times[-b_0,0],\label{Pi2.801}\ee
where $\a(\cdot,s)$ is a closed curve with  period $a$ for each $s\in[-b_0,b_1].$  Set
$$S=S^+\cup\Ga_0\cup S^-,$$ where
$$S^+=\{\,\a(t,s)\,|\,(t,s)\in[0,a)\times(0,b)\,\},\quad \Ga_0=\{\,\a(t,0)\,|\,t\in[0,a]\,\},$$
$$ S^-=\{\,\a(t,s)\,|\,(t,s)\in[0,a)\times(-b_0,0)\,\}.$$

 {\bf Curvature assumptions}\,\,\, Let $\kappa$ be the Gaussian curvature function on $M.$ We assume that $S$ satisfies the following curvature conditions:
\be\kappa(x)>0\qfq x\in S^+\cup\Ga_{b_1};\label{kappa+}\ee
\be\kappa=0,\quad D\kappa(x)\not=0\qfq x\in\Ga_0;\label{kappa0}\ee
\be\kappa(x)<0\qfq x\in S^-\cup\Ga_{-b_0},\label{kappa-}\ee where
$$\Ga_{b_1}=\{\,\a(t,b_1)\,|\,t\in[0,a)\,\},\quad \Ga_{-b_0}=\{\,\a(t,-b_0)\,|\,t\in[0,a)\,\}.$$

Our main results are the following.

\begin{thm}\label{t1.1}Let ${S}$ be  of class $\CC^{m+5}$  for some  integer $m\geq1.$ For $U\in\WW^{m+1,2}({S},T^2_{\sym}S),$ there exists a solution $y=W+w\n\in\WW^{m,2}({S},\R^3)$ to equation $(\ref{01})$ satisfying
$$ \|W\|_{\WW^{m+1,2}({S},TS)}^2+\|w\|_{\WW^{m,2}({S})}^2\leq C\|U\|_{\WW^{m+1,2}({S},T^2_{\sym}S)}^2.$$
\end{thm}

\begin{rem} The curvature of $S$ affects the regularity of $(\ref{01})$ as
$$\begin{cases} \|y\|^2_{\WW^{m,2}(S,\R^3)}\leq C\|U\|_{\WW^{m,2}({S},T^2_{\sym}S)}^2\quad\mbox{if $\kappa>0$ on $\overline{S}$};\\
\|y\|^2_{\WW^{m,2}(S,\R^3)}\leq C\|U\|_{\WW^{m+1,2}({S},T^2_{\sym}S)}^2\quad\mbox{if $\kappa<0$ on $\overline{S}$.}
\end{cases}$$
\end{rem}

By the imbedding theorem \cite[P. 158]{GNT}, the following corollary is immediate.

\begin{cor}\label{c1.1} Let $S$ be of ${S}\in\CC^{m+5}$ for some integer $m\geq0.$ Then problem $(\ref{01})$ admits a solution $y=W+w\n\in\CC^m_B(S,\R^3)$ satisfying
$$ \|W\|_{\CC^{m+1}_B({S},TS)}+\|w\|_{\CC^{m}_B(S)}\leq C\|U\|_{\CC^{m+3}_B({S},T^2_{\sym}S)},$$ where
$$\CC^m_B(S,\R^3)=\{\,y\in\CC^m(S,\R^3)\,|\,D^\a y\in\LL^\infty(S,\R^3)\,\,\mbox{for}\,\,|\a|\leq m\,\}.$$
\end{cor}

For $y\in\WW^{1,2}({S},\R^3),$ we denote the left hand side of equation (\ref{01}) by $\sym\nabla y.$
Let
$$\V({S},\R^3)=\{\,y\in\WW^{2,2}({S},\R^3)\,|\,\sym\nabla y=0\,\}.$$

\begin{thm}\label{t1.2} Let ${S}$ be  of class $\CC^{m+5,1}.$  Then, for every $y\in\V({S},\R^3)$
there exists a sequence $\{\,y_k\,\}\subset\V({S},\R^3)\cap \CC^m_B({S},\R^3)$ such that
$$\lim_{k\rw\infty}\|y-y_k\|_{\WW^{2,2}({S},\R^3)}=0.$$
\end{thm}

A one parameter family $\{\,y_\varepsilon\,\}_{\varepsilon>0} \subset\CC^1_B(\overline{{S}},\R^3)$ is said to be a (generalized)
$m$th order infinitesimal isometry if the change of metric induced by $y_\varepsilon$ is of order $\varepsilon^{m+1},$ that is,
$$\|\nabla^Ty_\varepsilon\nabla y_\varepsilon-g\|_{L^\infty({S},T^2)}=\O(\varepsilon^{m+1})\quad\mbox{as}\quad \varepsilon\rw0,$$ where $g$ is the induced metric of $M$ from $\R^3,$ see \cite{HoLePa}.
A given $m$th order infinitesimal isometry can be modified by higher order
corrections to yield an infinitesimal isometry of order $m_1>m,$ a property to which we
refer to by {\it matching property of infinitesimal isometries}, \cite{HoLePa,LeMoPa1}.

\begin{thm}\label{t1.3}Let ${S}$ be  of class $\CC^{4m,1}.$ Given $y\in\V({S},\R^3)\cap\CC^{4m-2}_B({S},\R^3),$
 there exists a family $\{\,z_\varepsilon\,\}_{\varepsilon>0}\subset\CC^2_B({S},\R^3),$ equi-bounded in $\CC^2_B({S},\R^3),$ such
that for all small $\varepsilon>0$  the family:
$$ y_\varepsilon=\id+\varepsilon y+\varepsilon^2z_\varepsilon$$
is a  $m$th order infinitesimal isometry of class $\CC^2_B({S},\R^3).$
\end{thm}

{\bf Application to elasticity of thin shells}\,\,\,
Let $\n$ be the normal field of surface $M.$ Consider a family $\{\,S_h\,\}_{h>0}$ of thin shells of thickness $h$ around $S,$
$$S_h=\{\,x+t\n(x)\,|\,x\in S,\,\,|t|<h/2\,\},\quad 0<h<h_0,$$ where $h_0$ is small enough so that the projection map $\pi:$  $S_h\rw S,$ $\pi(x+t\n)=x$ is well defined.
For a $\WW^{1,2}$ deformation $u_h:$ $S_h\rw\R^3,$ we assume that its elastic energy (scaled per
unit thickness) is given by the nonlinear functional:
$$E_h(u_h)=\frac{1}{h}\int_{S_h}W(\nabla u_h)dz.$$
The stored-energy density function $W:$ $\R^3\times\R^3\rw\R$ is $\CC^2$ in an open neighborhood of
SO(3), and it is assumed to satisfy the conditions of normalization, frame indifference and
quadratic growth: For all $F\in\R^3\times\R^3,$ $R\in\SO(3),$
$$W(R)=0,\quad W(RF)=W(F),\quad W(F)\geq C\dist^2(F,\SO(3)),$$
with a uniform constant $C>0.$ The potential $W$ induces the quadratic forms (\cite{FrJaMu})
$$\Q_3(F)=D^2W(Id)(F,F),\quad\Q_2(x,F_{\tan})=\min\{\,Q_3(\hat F)\,|\,\hat F=F_{\tan}\,\}.$$

Let $A:$ $S\rw\R^{3\times3}$ be a matrix field. We define $A\in T^2S$ by
$$A(\a,\b)=\<A(x)\a,\b\>\qfq\a,\,\,\b\in T_xS,\quad x\in S.$$
For given $V\in\V(\Om,\R^3),$ there exists a unique $A\in\WW^{1,2}(S,T^2)$ such that
\be\nabla_\a V=A(x)\a\qfq \a\in T_xS,\,\,\,\quad A(x)=-A^T(x),\quad x\in\Om.\label{1.5}\ee

We shall consider a sequence $e_h>0$ such that:
\be 0<\lim_{h\rw0}e_h/h^\b<\infty\quad\mbox{for some $2<\b\leq4.$}\label{x1.3}\ee
Let
$$\b_m=2+2/m.$$

Recall the following
results.
\begin{thm}\label{t1.4}$\cite{LeMoPa}$
Let $S$ be a surface embedded in $\R^3,$ which is compact, connected, oriented,
of class $\CC^{1,1},$ and whose boundary $\pl  S$ is the union of finitely many Lipschitz curves.
Let $u_h\in\WW^{1,2}(S_h,\R^3)$ be a sequence of deformations whose scaled energies $E_h(u_h)/e_h$ are
uniformly bounded. Then there exist a sequence $Q_h\in\SO(3)$ and $c_h\in\R^3$ such that for the
normalized rescaled deformations
$$y_h(z)=Q_hu_h(x+\frac{h}{h_0}t\n(x))-c_h,\quad z=x+t\n(x)\in S_{h_0},$$ the following holds.

$(i)$ $y_h$ converge to $\pi$ in $\WW^{1,2}(S_{h_0},\R^3).$

$(ii)$ The scaled average displacements
$$ V_h(x)=\frac{h}{h_0\sqrt{e_h}}\int_{-h_0/2}^{h_0/2}[y_h(x+t\n)-x]dt$$
converge to some $V\in\V(S,\R^3).$

$(iii)$ $\lim\inf_{h\rw0} E_h(u_h)/e_h\geq I(V),$ where
\be I(V)=\frac{1}{24}\int_S\Q_2\Big(x,(\nabla(A\n)-A\nabla\n)_{\tan}\Big)dg,\label{1.4}\ee where $A$ is given in $(\ref{1.5}).$
\end{thm}

The above result proves the lower bound for the $\Ga$-convergence. With Theorems \ref{t1.2} and \ref{t1.3} a recover sequence can be constructed in the $\Ga$-limit for thin shells as in \cite{HoLePa,LeMoPa1,Yao2017} such that the following theorem holds true. The details of the proof are omitted.

\begin{thm}\label{t1.5}
Let $S\subset M$ be  of class $\CC^{5,1}$ given in $(\ref{Q0})$ with $(\ref{Pi2.801})-(\ref{kappa-}).$ Suppose that $$e_h=\o(h^{\b_m}).$$
Then for every $V\in\V(S,\R^3)$ there exists a sequence of deformations $\{\,u_h\,\}\subset\WW^{1,2}(S,\R^3)$ such that $(i)$ and $(ii)$ of Theorem $\ref{t1.4}$ hold. Moreover,
\be \lim_{h\rw0}\frac{1}{e_h}E_h(u_h)=I(V),\label{lim}\ee where $I(V)$ is given in $(\ref{1.4}).$
\end{thm}

\setcounter{equation}{0}
\def\theequation{2.\arabic{equation}}
\section{$\LL^2$ Solutions of the tensor equation of mixed type }
\hskip\parindent

Let $\nabla$ and $D$ denote the connection of $\R^3$ in the Euclidean metric and the one of $M$ in the reduced metric,
respectively. We have to treat the relationship between
$\nabla$ and $D$ carefully.

Let $m\geq1$ be an integer. Let $T\in T^mM$ be a $m$th order tensor field on $M.$ We define a $m-1$th order tensor field by
$$\ii_YT(Y_1,\cdots,Y_{m-1})=T(Y,Y_1,\cdots,Y_{m-1})\qfq Y_1,\,\,\cdots,\,Y_{m-1}\in TM,$$ which is called an {\it inner product} of $T$ with $Y.$ For any $T\in T^2S$ and $\a\in T_xM,$
$$\tr_g\ii_\a  DT$$ is a linear functional on $T_xM,$ where $\tr_g\ii_\a DT$ is the trace of the 2-order tensor field $\ii_\a DT$ in the induced metric $g.$ Thus there is a vector, denoted by $\div_gT,$ such that
$$\<\div_gT,\a\>=\tr_g\ii_\a DT\qfq \a\in T_xM,\,\,x\in M.$$ Clearly, the above formula defines a vector field $\div_gT\in TM.$

We need a linear operator $Q$ (\cite{Yao2017}, \cite{Yao2018}) as follows. For each point
$p\in M$, the Riesz representation theorem implies that there
exists an isomorphism $Q:$ $T_pM\rw T_pM$ such that
\be \label{2.1}
\<\a,Q\b\>=\det\left(\a,\b,\vec{n}(p)\right)\qfq\a,\,\b\in T_pM.\ee
Let $e_1,$ $e_2$ be an orthonormal basis of $T_pM$ with
positive orientation, that is,
$$\det\Big(e_1,e_2,\n(p)\Big)=1.$$
Then $Q$ can be expressed explicitly by
\be Q\a=\<\a,e_2\>e_1-\<\a,e_1\>e_2\quad\mbox{for
all}\quad\a\in T_pM.\label{1.3n}\ee
Clearly, $Q$ satisfies
$$ Q^T=-Q,\quad Q^2=-\Id.$$
Operator $Q$ plays an important role in our analysis.

Here we sum up some formulas about operator $Q$ and their proofs have been given in \cite{CY}.
\begin{lem}\label{l2.1} The following identities hold.
\be QD_XY=D_X(QY)\qfq X,\,Y\in TM.\label{2.3}\ee
\be\<X,Y\>Z=\<Z,Y\>X+\<Z,QX\>QY\qfq X,\,Y,\,Z\in TM.\label{2.4}\ee

 Let $P\in T^2M.$
Let $X$ and $Y$ be vector fields and $f$ be a function. Then
\be
\div_g(PX)=\<P,DX\>+\<\div_gP,X\>,\label{2.5}\ee
$$\div_g(fP)=f\div_gP+P^TDf.$$
 Let $P\in T^2M$ and
let $p\in M$ be given. Then
\be
\<Qv,w\>P-\<Pv,w\>Q
=Qv\otimes Pw-Pv\otimes Qw,\label{2.6}
\ee
\be\<Qv,w\>QP+\<Pv,w\>\id=Qv\otimes QP w+Pv\otimes w \qfq v,\,\,w\in T_pM.\label{2.7}\ee

For any $X,$ $Y\in TS,$
\be\nabla\vec{n}[X,Y]=D_X\nabla\vec{n}Y-D_Y\nabla\vec{n}X,\label{2.92}\ee
\be \div_g[X,Y]=X\div_gY-Y\div_gX.\label{div[X,Y]}\ee
\end{lem}

\begin{lem} For any $Z,$ $X,$ $V,$ $W\in TM,$
\beq&&\<X,Z\>\<\<W,QX\>Q\nabla\n V+\<W,\nabla\n X\>V,\,\,Z\>\nonumber\\
&&=\Pi(X,X)\<V,Z\>\<W,Z\>-\Pi(QZ,QZ)\<V,QX\>\<W,QX\>\qfq p\in M.\label{2.8}\eeq
\end{lem}

\begin{proof}Using (\ref{2.4}) where $Y=Z$ and $Z=V,$ we have
$$\<X,Z\>V=\<V,Z\>X+\<V,QX\>QZ\qfq p\in M.$$ Similarly,
$$\<X,Z\>W=\<W,Z\>X+\<W,QX\>QZ\qfq p\in M.$$  Thus we obtain
\beq&&\<X,Z\>(\<W,QX\>Q\nabla\n V+\<W,\nabla\n X\>V)\nonumber\\
&&=\<W,QX\>Q\nabla\n(\<X,Z\> V)+\<\<X,Z\>W,\nabla\n X\>V\nonumber\\
&&=\<W,QX\>Q\nabla\n(\<V,Z\>X+\<V,QX\>QZ)+\<\<W,Z\>X+\<W,QX\>QZ,\,\,\nabla\n X\>V\nonumber\\
&&=\<V,Z\>\<W,QX\>Q\nabla\n X+\<V,QX\>\<W,QX\>Q\nabla\n QZ+\<W,Z\>\Pi(X,X)V\nonumber\\
&&\quad+\<W,QX\>\<QZ,\nabla\n X\>V\qfq p\in M.\nonumber\eeq
It follows by $Q^T=-Q$ and $\<QZ,Z\>=0$ that
\beq&&\<X,Z\>\<(\<W,QX\>Q\nabla\n V+\<W,\nabla\n X\>V),\,\,Z\>\nonumber\\
&&=\<V,Z\>\<W,QX\>\<Q\nabla\n X,Z\>+\<V,QX\>\<W,QX\>\<Q\nabla\n QZ,Z\>+\<W,Z\>\Pi(X,X)\<V,Z\>\nonumber\\
&&\quad-\<W,QX\>\<Z,Q\nabla\n X\>\<V,Z\>\nonumber\\
&&=\Pi(X,X)\<V,Z\>\<W,Z\>-\Pi(QZ,QZ)\<V,QX\>\<W,QX\>\qfq p\in M.\nonumber\eeq

\end{proof}

For given $\varepsilon>0$ small, let
$$\Om_{-\varepsilon}=\{\,\a(t,s)\,|\,(t,s)\in[0,a)\times(-\varepsilon,b_1)\,\},$$
$$\Ga_{-\varepsilon}=\{\,\a(t,-\varepsilon)\,|\,t\in[0,a)\,\},\quad \Ga_{b_1}=\{\,\a(t,b_1)\,|\,t\in[0,a)\,\}.$$
Let $X\in TS$ be a given vector field and $\nu$ be the outside normal of $\Om_{-\varepsilon}.$ We consider solvability of system
\be\left\{\begin{array}{l}\div_g(Q\nabla\n V)=f_1\qfq x\in \Om_{-\varepsilon},\\
\div_gV=f_2\qfq x\in \Om_{-\varepsilon},\\
\<W,QX\>|_{\Ga_{-\varepsilon}}=q,\quad\<W,\nu\>|_{\Ga_{b_1}}=p,\end{array}\right.\label{V2.10}\ee  where $V\in T\Om_{-\varepsilon}$ is the unknown, and $f_i\in\LL^2(\Om_{-\varepsilon}),$
$q\in\LL^2(\Ga_{-\varepsilon}),$ and $p\in\LL^2(\Ga_{b_1})$ are given.

Let  $Y\in TS$ be a given vector field. In order to solve problem (\ref{V2.10}), we introduce an operator $\L_Y:$ $\LL^2(S,T)\rw\LL^2(S,T)$ by
\be \L_YV
=e^{-\g\kappa}[(\div_g Q\nabla\n V+\<V,Y-\eta QX\>)QX+(\div_g V-\eta\<V,\nabla\n X\>)\nabla\n X],\label{2.10}\ee
\be D(\L_Y)=\{\,V\in \LL^2(S,T)\big|\,\div_g V\in \LL^2(S), \div_gQ\nabla\n V\in \LL^2(S)\,\},\ee
where $\g\in[0,\infty)$ and $\eta\in\CC^m(\Om_{-\varepsilon}).$

For $V,$ $W\in\WW^{1,2}(S,T),$ we compute
\beq&&\<W,\L_YV\>=e^{-\g\kappa}[(\div_g Q\nabla\n V+\<V,Y-\eta QX\>)\<W,QX\>\nonumber\\
&&\quad+(\div_g V-\eta\<V,\nabla\n X\>)\<W,\nabla\n X\>]\nonumber\\
&&=\div_g(e^{-\g\kappa}\<W,QX\>Q\nabla\n V)-\<Q\nabla\n V,D(e^{-\g\kappa}\<W,QX\>)\>+\div_g(e^{-\g\kappa}\<W,\nabla\n X\>V)\nonumber\\
&&\quad-\<V,D(e^{-\g\kappa}\<W,\nabla\n X\>)\>+e^{-\g\kappa}(\<V,Y-\eta QX\>\<W,QX\>-\eta\<V,\nabla\n X\>\<W,\nabla\n X\>)\nonumber\\
&&=\div_g[e^{-\g\kappa}(\<W,QX\>Q\nabla\n V+\<W,\nabla\n X\>V)]\nonumber\\
&&\quad+\g e^{-\g\kappa}(\<W,QX\>\<Q\nabla\n V,D\kappa\>+\<W,\nabla\n X\>\<V,D\kappa\>)-e^{-\g\kappa}(\<Q\nabla\n V, D\<W,QX\>\>\nonumber\\
&&\quad+\<V,D\<W,\nabla\n X\>\>)+e^{-\g\kappa}(\<V,Y-\eta QX\>\<W,QX\>-\eta\<V,\nabla\n X\>\<W,\nabla\n X\>)\nonumber\\
&&=\div_g[e^{-\g\kappa}(\<W,QX\>Q\nabla\n V+\<W,\nabla\n X\>V)]\nonumber\\
&&\quad-\g e^{-\g\kappa}(\<V,\<W,QX\>\nabla\n QD\kappa-\<W,\nabla\n X\>D\kappa\>)\nonumber\\
&&\quad+e^{-\g\kappa}\<V, \nabla\n QD\<W,QX\>-D\<W,\nabla\n X\>\>\nonumber\\
&&\quad+e^{-\g\kappa}\<V,\<W,QX\>(Y-\eta QX)-\eta\<W,\nabla\n X\>\nabla\n X\>\nonumber\\
&&=\div_g[e^{-\g\kappa}(\<W,QX\>Q\nabla\n V+\<W,\nabla\n X\>V)]+\<V,\L_Y^*W\>,\label{2.12}\eeq where
\beq\L_Y^*W&&=e^{-\g\kappa}[\nabla\n QD\<W,QX\>-D\<W,\nabla\n X\>-\g(\<W,QX\>\nabla\n QD\kappa-\<W,\nabla\n X\>D\kappa)\nonumber\\
&&\quad+\<W,QX\>(Y-\eta QX)-\eta\<W,\nabla\n X\>\nabla\n X].\label{2.13}\eeq

\begin{lem} Let $X,$ $Y\in TS$ be given. Suppose that $\Om\subset S$ be a subregion of $S$ such that
$$\<X,\nu\>\not=0\qfq x\in\pl\Om,$$ where $\nu$ is the outside normal of $\Om.$  Then
\beq (W,\L_YV)_{\LL^2(\Om,T)}&&=(V,\L_Y^*W)_{\LL^2(\Om,T)}\nonumber\\
&&\quad+\int_{\pl\Om}(\pounds_1\<V,\nu\>\<W,\nu\>-\pounds_2\<V,QX\>\<W,QX\>)d\Ga,\label{2.14}\eeq
for $V,$ $W\in T\Om,$ where
\be\pounds_1=\frac{e^{-\g\kappa}}{\<X,\nu\>}\Pi(X,X),\quad \pounds_2=\frac{e^{-\g\kappa}}{\<X,\nu\>}\Pi(Q\nu,Q\nu) \qfq x\in\pl\Om.\label{n2.17x}\ee
\end{lem}

\begin{proof} We integrate (\ref{2.12}) over $\Om$ and use (\ref{2.8}), where $Z=\nu,$ to have (\ref{2.14}).
\end{proof}

\begin{lem} Let $X,$ $Y\in TS$ be given. Let $\L_Y$ and $\L_Y^*$ be given in $(\ref{2.10})$ and $(\ref{2.13}),$ respectively. Then
\beq&&-e^{\g\kappa}\<W,\L_YW+\L_Y^*W\>\nonumber\\
&&=\g(\<W,QX\>\<W,\nabla\n QD\kappa\>-\<W,\nabla\n X\>\<W, D\kappa\>)+\<W, D_W(\nabla\n X)\>\nonumber\\
&&\quad+2\eta(\<W,QX\>^2+\<W,\nabla\n X\>^2)-\<W,2Y+\div_gQ\nabla\n\>\<W,QX\>\nonumber\\
&&\quad+\<W,\,\,D_{Q\nabla\n W}(QX)\>\qfq x\in \Om_{-\varepsilon}.\label{2.16}\eeq
\end{lem}

\begin{proof} From (\ref{2.10}) and (\ref{2.5}), we have
\beq e^{-\g\kappa}\<W,\L_YW\>&&=(\div_g Q\nabla\n W+\<W,Y-\eta QX\>)\<W,QX\>\nonumber\\
&&\quad+(\div_g W-\eta\<W,\nabla\n X\>)\<W,\nabla\n X\>\nonumber\\
&&=\<DW,\,\,\<W,QX\>Q\nabla\n+\<W,\nabla\n X\>\id\>-\eta(\<W,QX\>^2+\<W,\nabla\n X\>^2)\nonumber\\
&&\quad+\<W,Y+\div_gQ\nabla\n\>\<W,QX\>\qfq p\in \Om_{-\varepsilon},\label{2.17}\eeq where formula $\div_gW=\<DW,\id\>$ has been used.
In addition, it follows from (\ref{2.13}) that
\beq e^{-\g\kappa}\<W,\L_Y^*W\>&&=\<W,\,\,\nabla\n QD\<W,QX\>-D\<W,\nabla\n X\>\>\nonumber\\
&&\quad-\g(\<W,QX\>\<W,\nabla\n QD\kappa\>-\<W,\nabla\n X\>\<W, D\kappa\>)\nonumber\\
&&\quad+\<W,QX\>\<W,\,Y-\eta QX\>-\eta\<W,\nabla\n X\>^2\nonumber\\
&&=-\<DW,\,\,QX\otimes Q\nabla\n W+\nabla\n X\otimes W\>-\<W,\,\,D_{Q\nabla\n W}(QX)+D_W(\nabla\n X)\>\nonumber\\
&&\quad-\g(\<W,QX\>\<W,\nabla\n QD\kappa\>+\<W,\nabla\n X\>\<W, D\kappa\>)\nonumber\\
&&\quad+\<W,QX\>\<W,\,Y\>-\eta(\<W, QX\>^2+\<W,\nabla\n X\>^2)\qfq p\in \Om_{-\varepsilon}.\label{2.18}\eeq Moreover, formula (\ref{2.7}) yields
$$\<W,QX\>Q\nabla\n+\<W,\nabla\n X\>\id=QX\otimes Q\nabla\n W+\nabla\n X\otimes W.$$
Thus (\ref{2.16}) follows from (\ref{2.17}) and (\ref{2.18}).
\end{proof}

For given $\varepsilon>$ small, set
$$S_\varepsilon=\{\,\a(t,s)\in \,|\,(t,s)\in[0,a)\times(-\varepsilon,\varepsilon)\,\}.$$
Let $x\in\Ga_0.$ Since $\kappa(x)=0$ and $D\kappa(x)\not=0,$ from \cite[Lemma 2.6]{Yao2020}, there exist vector fields $X_1,$ $X_2$ in a neighborhood of $x$
satisfying $\nabla\n X_i=\lam_iX_i,$ where $\lam_i$ are the principal curvatures. Clearly we may extend the vector fields $X_i$ to the region $\overline{S}_{\varepsilon}$ when $\varepsilon>0$ is given small.
We assume that $X_i$ are vector fields such that
\be\nabla\n X_i=\lam_iX_i,\quad|X_i|=1,\quad\<X_1,X_2\>=0\qfq x\in\overline{S}_{\varepsilon},\label{X2.82}\ee where
$$\lam_1>0\qfq x\in S_\varepsilon,\quad\lam_2=0\qfq x\in\Ga_0,$$
$$\lam_2>0\qfq x\in S^+\cap \overline{S}_\varepsilon,\quad \lam_2<0\qfq x\in S^-\cap\overline{S}_{\varepsilon}.$$

\begin{lem} For given $\varepsilon>0$ small
$$X_2(\lam_2)\not=0\qfq x\in S_{\varepsilon}.$$
\end{lem}

\begin{proof} It will suffice to prove
$$X_2(\lam_2)\not=0\qfq x\in\Ga_0.$$

First, we claim that $\<Q\a_t,X_2\>(x)\not=0$ for $x\in\Ga_0.$ If not, then $\<Q\a_t,X_2\>(x)=0$ implies that $X_2=\eta \a_t$ with $\eta\not=0,$ and thus
$$\eta^2\Pi(\a_t,\a_t)=\<\nabla_{X_2}\n,X_2\>=\lam_2(x)=0,$$ which contradicts  assumption $\Pi(\a_t,\a_t)\not=0.$

In addition, assumption (\ref{kappa0}) implies
\be\<D\kappa,\a_t\>=0\qfq x\in\Ga_0,\quad D\kappa=\iota Q\a_t\qfq x\in\Ga_0,\label{Pii}\ee   for some $\iota\not=0.$
 Thus we obtain
\be X_2(\lam_2)=\frac1{\lam_1}X_2(\kappa)=\frac\iota{\lam_1}\<X_2,Q\a_t\>\not=0\qfq x\in\Ga_0.\label{X21}\ee
\end{proof}

From (\ref{Pi2.801}) and (\ref{kappa+}), we have
\be\Pi(\a_t,\a_t)>0\qfq x\in\overline{S}.\label{Pi2.80}\ee
We assume that for each $s\in[-b_0,b_1]$ the closed curve $\a(\cdot,s)$ goes in anticlockwise with $t\in[0,a)$ increasing.
We further assume that
\be X_2(\lam_2)>0\qfq\overline{S}_\varepsilon,\label{as}\ee for given $\varepsilon>0$ small. For otherwise, we replace $X_2$ with $-X_2.$
Furthermore, we assume that $X_1,$ $X_2$ has positive orientation. For otherwise, we replace $X_1$ with $-X_1.$ Thus
\be QX_2=X_1,\quad QX_1=-X_2\qfq x\in\overline{S}_\varepsilon.\label{2.76}\ee

\begin{lem}\label{l2.6} We have
\be \left\{\begin{array}{l}\<X_2,Q\a_t\><0\quad\mbox{if}\quad\det(\a_t,\a_s,\n)>0,\\
\<X_2,Q\a_t\>>0\quad\mbox{if}\quad\det(\a_t,\a_s,\n)<0\end{array}\right.\qfq x\in\Ga_0.\label{2.26}\ee
\end{lem}

\begin{proof} Let $\det(\a_t,\a_s,\n)<0.$ The case of $\det(\a_t,\a_s,\n)>0$ can be treated similarly.

Since
$$\frac{Q\a_t}{|\a_t|},\quad \frac{\a_t}{|\a_t|}$$ forms an orthonormal basis with positive orientation,  the curvature assumptions (\ref{kappa+})-(\ref{kappa-}) yield
$$\<Q\a_t,D\kappa\>>0\qfq x\in\Ga_0.$$
It follows from (\ref{Pii}) that $\iota>0$ for $x\in\Ga_0.$ Thus in the case of $\det(\a_t,\a_s,\n)<0$ (\ref{2.26}) follows from (\ref{as}) and (\ref{X21}).
\end{proof}

We set
\be X_{\si_0}=\left\{\begin{array}{l} \si_0\a_t-Q\a_t \quad\mbox{if}\quad\det(\a_t,\a_s,\n)>0,\\
  \si_0\a_t+Q\a_t \quad\mbox{if}\quad\det(\a_t,\a_s,\n)<0, \end{array}\right.\label{2.27}\ee where $\si_0$ is a constant satisfying
$$\si_0\geq1+\sup_{x\in\Ga_0}\frac{|\Pi(\a_t,Q\a_t)|}{\Pi(\a_t,\a_t)}.$$

\begin{lem}
We have
$$\Pi(X_{\si_0},X_{\si_0})>0\qfq x\in\overline{S^+},$$
\be\<X_{\si_0},\nu\><0\qfq x\in\Ga_0,\quad \<X_{\si_0},\nu\>>0\qfq x\in\Ga_{b_1},\label{2.28}\ee where $\nu$ is the outside normal of region $S^+.$
\end{lem}

\begin{proof} It is easy to check that
$$\left\{\begin{array}{l} \nu=\frac{Q\a_t}{|\a_t|}\qfq x\in\Ga_0,\\
\nu=-\frac{Q\a_t}{|\a_t|}\qfq x\in\Ga_{b_1}, \end{array}\right.\quad\mbox{if}\quad \det(\a_t,\a_s,\n)>0$$ and
$$\left\{\begin{array}{l} \nu=-\frac{Q\a_t}{|\a_t|}\qfq x\in\Ga_0,\\
\nu=\frac{Q\a_t}{|\a_t|}\qfq x\in\Ga_{b_1}, \end{array}\right.\quad\mbox{if}\quad \det(\a_t,\a_s,\n)<0.$$
Thus (\ref{2.28}) follows. Moreover, we have
$$\Pi(X_{\si_0},X_{\si_0})\geq\Pi(\a_t,\a_t)(\si_0\pm\frac{\Pi(\a_t,Q\a_t)}{\Pi(\a_t,\a_t)})^2>0\qfq x\in\Ga_0.$$
\end{proof}

By  Lemma \ref{l2.6} we fix $\varepsilon>0$ small such that
\be \left\{\begin{array}{l}\<X_2,Q\a_t\><0\quad\mbox{if}\quad\det(\a_t,\a_s,\n)>0,\\
\<X_2,Q\a_t\>>0\quad\mbox{if}\quad\det(\a_t,\a_s,\n)<0\end{array}\right.\qfq x\in\Ga_{-\varepsilon}.\label{2.29}\ee  Moreover, since
$$\<\a_t,X_1\>\not=0\qfq x\in\Ga_0,$$ we assume that the above $\varepsilon$ is such that
$$\<\a_t,X_1\>\not=0\qfq x\in\overline{S}_{2\varepsilon}.$$
Then
\be\left\{\begin{array}{l} \nu=\frac{Q\a_t}{|\a_t|}\qfq x\in\Ga_{-\varepsilon},\quad\mbox{if}\quad \det(\a_t,\a_s,\n)>0,\\
\nu=-\frac{Q\a_t}{|\a_t|}\qfq x\in\Ga_{-\varepsilon}, \quad\mbox{if}\quad \det(\a_t,\a_s,\n)>0.\end{array}\right.\label{nu2.31}\ee
Thus, by (\ref{2.29}),
\be\<X_2,\nu\><0\qfq x\in\Ga_{-\varepsilon}.\label{nu2.30}\ee

We take  $\zeta\in\CC_0^\infty(-\infty,\infty)$ such that
$$0\leq\zeta\leq1;\quad\zeta=1\qfq s\in(-\varepsilon,\varepsilon);\quad \zeta=0\qfq s\geq2\varepsilon.$$ Then we take
\be \eta=s\eta_0(s)\label{2.32*}\ee in (\ref{2.10})  where
$\eta_0\in\CC^\infty_0(-\infty,\infty)$  such that
$$0\leq\eta_0\leq1;\quad \eta_0=0\qfq s\leq \varepsilon/2;\quad \eta_0=1\qfq s\geq\varepsilon.$$ In addition, we take $X$ in (\ref{2.10}) by
\be X=\zeta[\varepsilon^{1/3}(\sign\<\a_t,X_1\>) X_1+X_2]+(1-\zeta)X_{\si_0}\qfq x\in\Om_{-\varepsilon},\label{2.32}\ee where $\sign\<\a_t,X_1\>$ is the sign function for $x\in\overline{S}_{2\varepsilon},$ $X_{\si_0}$ is given in (\ref{2.27}), and $\si_0$ is given such that
\be\si_0\geq1+\sup_{x\in\Ga_0}\frac{|\Pi(\a_t,Q\a_t)|}{\Pi(\a_t,\a_t)}+\sup_{x\in\overline{S}_{2\varepsilon}}\frac{|\<Q\a_t,X_1\>|}{|\<\a_t,X_1\>|}.\label{2.33}\ee Then
\be\Pi(X,X)>0,\quad \<X,\nu\><0\qfq x\in\Ga_{-\varepsilon},\label{2.35x}\ee when $\varepsilon>0$ is small enough.

\begin{lem} Let $\eta$ and $X$ be  given in $(\ref{2.32*})$ and $(\ref{2.32}),$ respectively. Let $Y\in\CC^0(\overline{\Om}_{-\varepsilon},T).$  Then there are constants $\si>0,$ $\g>0,$ and $s>0$  such that
\be -\<W,\,\,\L_YW+\L_Y^*W\>\geq\si|W|^2\qfq  W\in T\Om_{-\varepsilon},\quad x\in\Om_{-\varepsilon},\label{2.34} \ee for given $\varepsilon>0$ small.
\end{lem}

\begin{proof} {\bf Step 1.}\,\,\, First, we prove
\be |X|>0\qfq x\in\overline{S^+\setminus S_\varepsilon}.\ee If not, there were a point $x\in\overline{S^+\setminus S_\varepsilon}$ such that
$$X=0.$$
If $\zeta=0,$ then
$$|X|=|X_{\si_0}|\geq\si_0|\a_t|>0,$$ which is not possible. Thus $\zeta\not=0.$ By (\ref{2.33}),
$$0=\<X,X_1\>=(1-\zeta)(\si_0\<\a_t,X_1\>\mp\<Q\a_t,X_1\>)+\zeta\varepsilon\sign\<\a_t,X_1\>$$ implies that
 $$\zeta=\frac{\si_0\<\a_t,X_1\>\mp\<Q\a_t,X_1\>}{\si_0\<\a_t,X_1\>\mp\<Q\a_t,X_1\>-\varepsilon\sign\<\a_t,X_1\>}>1,$$ which contradicts $0\leq\zeta\leq1$ when given $\varepsilon>0$ is small.

Since $\kappa>0$ for $x\in\overline{S^+\setminus S_\varepsilon},$ $X,$ $\nabla\n X$ forms an vector field basis on $\overline{S^+\setminus S_\varepsilon}.$ Then there exists $\si>0$ such that
\be\<W,QX\>^2+\<W,\nabla\n X\>^2\geq\si|W|^2\qfq x\in\overline{S^+\setminus S_\varepsilon}.  \label{QX2.36} \ee

{\bf Step 2.}\,\,\,Consider the region $S_\varepsilon.$ Noting that $\zeta=1$ for $x\in S_\varepsilon,$ we have
  $$X=\varepsilon^{1/3}(\sign\<\a_t,X_1\>) X_1+X_2\qfq x\in\overline{S}_\varepsilon.$$ Denote $\<W,X_i\>=W_i$ for $i=1,$ $2.$ We compute
\beq&&\g(\<W,QX_2\>\<W,\nabla\n QD\kappa\>-\<W,\nabla\n X_2\>\<W, D\kappa\>)+\<W, D_W(\nabla\n X_2)\>\nonumber\\
&&\quad-\<W,2Y+\div_gQ\nabla\n\>\<W,QX_2\>+\<W,\,\,D_{Q\nabla\n W}(QX_2)\>\nonumber\\
&&=\g(\<W, X_1\>\<W,\nabla\n QD\kappa\>-\lam_2\<W,X_2\>\<W, D\kappa\>)+\<W, D_W(\lam_2 X_2)\>\nonumber\\
&&\quad-\<W,2Y+\div_gQ\nabla\n\>\<W,X_1\>+\<W,\,\,D_{Q\nabla\n W}(X_1)\>\nonumber\\
&&\geq\g[\lam_1X_2(\kappa)-|\lam_2||X_1(\kappa)|]W_1^2+[X_2(\lam_2)-|\lam_2|(|\<X_2,D_{X_1}X_1\>|+\g X_2(\kappa))]W_2^2\nonumber\\
&&\quad-C(W_1^2+|W_1||W_2|)\qfq x\in\overline{S}_\varepsilon.\label{2.37}\eeq Similarly, we have
\beq&&\g(\<W,QX_1\>\<W,\nabla\n QD\kappa\>-\<W,\nabla\n X_1\>\<W, D\kappa\>)+\<W, D_W(\nabla\n X_1)\>\nonumber\\
&&\quad-\<W,2Y+\div_gQ\nabla\n\>\<W,QX_1\>+\<W,\,\,D_{Q\nabla\n W}(QX_1)\>\nonumber\\
&&\geq-\g\lam_1(|X_1(\kappa)|+X_2(\kappa))W_1^2-C(W_1^2+|W_1||W_2|)\nonumber\\
&&\quad-[\g(|\lam_2||X_1(\kappa)|+\lam_1X_2(\kappa))+|\<X_2,D_{X_2}(\lam_1X_1)\>|+|\<\div_gQ\nabla\n,X_2\>|]W_2^2.\quad\quad\label{L2.38}\eeq
It follows from (\ref{2.16}), (\ref{2.37}), and (\ref{L2.38}) that
\beq&&-\<W,\,\,\L_YW+\L_Y^*\>\geq\g[\lam_1X_2(\kappa)-|\lam_2||X_1(\kappa)|-\varepsilon\lam_1(|X_1(\kappa)|+X_2(\kappa))]W_1^2\nonumber\\
&&\quad+\{X_2(\lam_2)-|\lam_2|(|\<X_2,D_{X_1}X_1\>|+\g X_2(\kappa))\nonumber\\
&&\quad-\varepsilon[\g(|\lam_2||X_1(\kappa)|+\lam_1X_2(\kappa))+|\<X_2,D_{X_2}(\lam_1X_1)\>|+|\<\div_gQ\nabla\n,X_2\>|]\}W_2^2\nonumber\\
&&\quad-C(W_1^2+|W_1||W_2|)\nonumber\\
&&\geq\{\g[\lam_1X_2(\kappa)-|\lam_2||X_1(\kappa)|-\varepsilon\lam_1(|X_1(\kappa)|+X_2(\kappa))]-C\}W_1^2\nonumber\\
&&\quad+\{[X_2(\lam_2)-|\lam_2|(|\<X_2,D_{X_1}X_1\>|+\g X_2(\kappa))]/2\nonumber\\
&&\quad-\varepsilon[\g(|\lam_2||X_1(\kappa)|+\lam_1X_2(\kappa))+|\<X_2,D_{X_2}(\lam_1X_1)\>|+|\<\div_gQ\nabla\n,X_2\>|]\}W_2^2,\quad\quad\quad\label{2.39}\eeq for $x\in\overline{S}_\varepsilon,$
 where $C>0$ is independent of $W\in TS_\varepsilon.$ Noting that $X_2(\kappa)=\lam_1X_2(\lam_2)>0$ for $x\in\Ga_0$ and
$$|\lam_2|=\O(\varepsilon)\qfq x\in \overline{S}_\varepsilon,$$ we assume that the $\varepsilon>0$ has been given so small such that
$$\lam_1X_2(\kappa)-|\lam_2||X_1(\kappa)|-\varepsilon\lam_1(|X_1(\kappa)|+X_2(\kappa))>0\qfq x\in\overline{S}_\varepsilon.$$ Then we fix $\g>0$ such that
$$\g[\lam_1X_2(\kappa)-|\lam_2||X_1(\kappa)|]-\varepsilon\lam_1(|X_1(\kappa)|+X_2(\kappa))-C>0\qfq x\in\overline{S}_\varepsilon.$$
Finally, we move the $\varepsilon>0$ to zero again such that
\beq&&\{[X_2(\lam_2)-|\lam_2|(|\<X_2,D_{X_1}X_1\>|+\g X_2(\kappa))]/2\nonumber\\
&&\quad-\varepsilon[\g(|\lam_2||X_1(\kappa)|+\lam_1X_2(\kappa))+|\<X_2,D_{X_2}(\lam_1X_1)\>|+|\<\div_gQ\nabla\n,X_2\>|]\}>0\nonumber\eeq for $x\in\overline{S}_\varepsilon.$
Thus for the above given $\g>0$ and  $\varepsilon>0,$ from (\ref{2.39}),  we have obtained $\si>0$ such that
\be-e^{\g\kappa}\<W,\L_YW+\L_Y^*W\>\geq\si|W|^2\qfq W\in T\overline{S}_\varepsilon,\quad x\in\overline{S}_\varepsilon.\label{2.38}\ee

{\bf Step 3.}\,\,\,Consider region $S^+\setminus S_\varepsilon.$ From (\ref{2.16}) and (\ref{QX2.36}), we have
\beq-e^{\g\kappa}\<W,\L_YW+\L_Y^*W\>&&\geq s\eta_0(\<W,QX\>^2+\<W,\nabla\n X\>^2)-C|W|^2\nonumber\\
&&\geq (s-C)|W|^2\qfq s>0, \quad W\in TS^+\setminus S_\varepsilon,\label{2.399}\eeq for $x\in\overline{S^+\setminus S_\varepsilon}.$
Thus (\ref{2.34}) follows from (\ref{2.38}) and (\ref{2.399}) after we take $s=C+\si$ in (\ref{2.399}).
\end{proof}

Next, we consider solvability of problem (\ref{V2.10}) on region $\Om_{-\varepsilon}.$ Let $F\in\LL^2(\Om_{-\varepsilon},T),$  $p\in\LL^2(\Ga_{b_1}),$ and $q\in\LL^2(\Ga_{-\varepsilon})$ be given. Clearly, system (\ref{V2.10}) is equivalent to problem
\be\left\{\begin{array}{l}\L_Y V=F\qfq x\in\Om_{-\varepsilon},\\
\<V,QX\>=q\qfq x\in\Ga_{-\varepsilon},\quad \<V,\nu\>=p\qfq x\in\Ga_{b_1}, \end{array}\right.\label{2.40}\ee where $\L_Y$ is defined in (\ref{2.10}) and vector field $X$ is given in (\ref{2.32}).

Let $m\geq0$ be an integer. Set
$$\CC^m_{-\varepsilon,b_1}(\Om_{-\varepsilon},T)=\{\,W\in\CC^m(\Om_{-\varepsilon},T)\,|\,\<W,QX\>|_{\Ga_{-\varepsilon}}=\<W,\nu\>|_{\Ga_{b_1}}=0\,\},$$
$$\CC^m_{b_1,-\varepsilon}(\Om_{-\varepsilon},T)=\{\,W\in\CC^m(\Om_{-\varepsilon},T)\,|\,\<W,QX\>|_{\Ga_{b_1}}=\<W,\nu\>|_{\Ga_{-\varepsilon}}=0\,\}.$$
We denote the completions of $\CC^m(\Om_{-\varepsilon},T)$ by $\LL^2_{-\varepsilon,b_1}(\Om_{-\varepsilon},T)$ and $\LL^2_{b_1,-\varepsilon}(\Om_{-\varepsilon},T)$ in the norms of
$$\|W\|^2_{\LL^2_{-\varepsilon,b_1}(\Om_{-\varepsilon},T)}=\|W\|^2_{\LL^2(\Om_{-\varepsilon},T)}+\|\<W,QX\>\|^2_{\LL^2(\Ga_{-\varepsilon})}+\|\<W,\nu\>\|^2_{\LL^2(\Ga_{b_1})}$$
and
$$\|W\|^2_{\LL^2_{b_1,-\varepsilon}(\Om_{-\varepsilon},T)}=\|W\|^2_{\LL^2(\Om_{-\varepsilon},T)}+\|\<W,QX\>\|^2_{\LL^2(\Ga_{b_1})}+\|\<W,\nu\>\|^2_{\LL^2(\Ga_{-\varepsilon})},$$
respectively.

\begin{thm}\label{t2.1} Let vector field $X$ be given in $(\ref{2.32})$ and $Y\in\CC^0(\overline{\Om}_\varepsilon, T).$  Let $F\in\LL^2(\Om_{-\varepsilon},T),$  $p\in\LL^2(\Ga_{b_1}),$ and $q\in\LL^2(\Ga_{-\varepsilon})$ be given. Then problem $(\ref{2.40})$ admits a unique solution $V\in D(\L_Y)\cap\LL^2_{-\varepsilon,b_1}(\Om_{-\varepsilon},T),$ which satisfies
\beq(V,\L_Y^*W)_{\LL^2(\Om_{-\varepsilon},T)}&&=(W,F)_{\LL^2(\Om_{-\varepsilon},T)}+\int_{\Ga_{-\varepsilon}}\pounds_2q\<W,QX\>d\Ga-\int_{\Ga_{b_1}}\pounds_1p\<W,\nu\>d\Ga\quad\quad\quad\label{2.411}\eeq for $W\in \CC^1_{b_1,-\varepsilon}(\Om_{-\varepsilon},T),$ where $\pounds_1$ and $\pounds_2$ are given in $(\ref{n2.17x}).$ Furthermore, the following estimate holds true.
\beq\|V\|^2_{\LL^2(\Om_{-\varepsilon},T)}&&+\|V\|^2_{\LL^2(\Ga_{-\varepsilon},T)}+\|V\|^2_{\LL^2(\Ga_{b_1},T)}\nonumber\\
&&\quad\leq C(\|F\|^2_{\LL^2(\Om_{-\varepsilon},T)}+\|q\|^2_{\LL^2(\Ga_{-\varepsilon})}+\|p\|^2_{\LL^2(\Ga_{b_1})}).\label{V2.45}\eeq
\end{thm}

\begin{proof} {\bf Step 1.}\,\,\,Denote the completion of $\CC^1_{b_1,-\varepsilon}(\Om_{-\varepsilon},T)$ in the norm of $\WW^{1,2}(\Om_{-\varepsilon},T)$ by $\WW^{1,2}_{b_1,-\varepsilon}(\Om_{-\varepsilon},T).$

We define a bilinear functional,
$$\a:\quad\LL^2_{-\varepsilon,b_1}(\Om_{-\varepsilon},T)\times\WW^{1,2}_{b_1,-\varepsilon}(\Om_{-\varepsilon},T)\rw \R,$$ by
$$\a(V,W)=-2(V,\L_Y^*W)_{\LL^2(\Om_{-\varepsilon},T)}.$$ Clearly, we have
$$|\a(V,W)|\leq C\|V\|_{\LL^2(\Om_{-\varepsilon},T)}\|W\|_{\WW^{1,2}(\Om_{-\varepsilon},T)}$$ for $(V,W)\in\LL^2_{-\varepsilon,b_1}(\Om_{-\varepsilon},T)\times\WW^{1,2}_{b_1,-\varepsilon}(\Om_{-\varepsilon},T).$ On the other hand, it follows from (\ref{2.28}) and (\ref{nu2.30}) that
\be\<X,\nu\><0\qfq x\in\Ga_{-\varepsilon},\quad \<X,\nu\>>0\qfq x\in\Ga_{b_1}.\label{X2.45}\ee
Thus,
from (\ref{2.14}) and (\ref{2.34}), we have
\beq &&\a(W,W)=-(W,\L_YW+\L_Y^*W)_{\LL^2(\Om_{-\varepsilon},T)}+\int_{\pl\Om_{-\varepsilon}}(\pounds_1\<W,\nu\>^2-\pounds_2\<W,QX\>^2)d\Ga\nonumber\\
&&\geq\si\|W\|^2_{\LL^2(\Om_{-\varepsilon}, T)}-\int_{\Ga_{-\varepsilon}}\pounds_2\<W,QX\>^2d\Ga+\int_{\Ga_{b_1}}\pounds_1\<W,\nu\>^2d\Ga\nonumber\\
&&\geq\si\|W\|^2_{\LL^2_{-\varepsilon,b_1}(\Om_{-\varepsilon},T)}\qfq W\in\CC^1_{b_1,-\varepsilon}(\Om_{-\varepsilon},T).\label{B2.47}\eeq

Next, we define a functional on $\LL^2_{-\varepsilon,b_1}(\Om_{-\varepsilon},T)$ by
$$\F(W)=-2(W,F)_{\LL^2(\Om_{-\varepsilon},T)}-2\int_{\Ga_{-\varepsilon}}\pounds_2q\<W,QX\>d\Ga+2\int_{\Ga_{b_1}}\pounds_1p\<W,\nu\>d\Ga.$$
clearly, $\F$ is bounded on $\LL^2_{-\varepsilon,b_1}(\Om_{-\varepsilon},T).$  By Theorem A given in Appendix there exists a $V\in \LL^2_{-\varepsilon,b_1}(\Om_{-\varepsilon},T)$
such that
\beq-2(V,\L_Y^*W)_{\LL^2(\Om_{-\varepsilon},T)}&&=-2(W,F)_{\LL^2(\Om_{-\varepsilon},T)}-2\int_{\Ga_{-\varepsilon}}\pounds_2q\<W,QX\>d\Ga\nonumber\\
&&\quad+2\int_{\Ga_{b_1}}\pounds_1p\<W,\nu\>d\Ga\label{F2.48}\eeq
for all $W\in\CC^1_{b_1,-\varepsilon}(\Om_{-\varepsilon},T),$ which yield, by (\ref{2.14}) again,
\beq &&-2(W,\L_YV)_{\LL^2(\Om_{-\varepsilon},T)}=-2(W,F)_{\LL^2(\Om_{-\varepsilon},T)}\nonumber\\
&&+2\int_{\Ga_{-\varepsilon}}\pounds_2(\<V,QX\>-q)\<W,QX\>d\Ga+2\int_{\Ga_{b_1}}\pounds_1(p-\<V,\nu\>)\<W,\nu\>d\Ga\nonumber\eeq
for all $W\in\CC^1_{b_1,-\varepsilon}(\Om_{-\varepsilon},T).$ Thus $V\in\LL^2_{-\varepsilon,b_1}(\Om_{-\varepsilon},T)$ solves problem (\ref{2.40}) and identity (\ref{2.411}) holds.

Noting that $\zeta=1$ for $x\in\overline{S}_\varepsilon,$ from (\ref{2.32}), we have
$$|\nabla\n X|^2=\varepsilon^2\lam_1^2+\lam_2^2>0\qfq x\in\overline{S}_\varepsilon.$$ Then, by (\ref{QX2.36}), $QX,$ $\nabla\n X$ forms a vector field basis on $\overline{\Om}_\varepsilon.$ Thus
$V\in D(\L_Y).$

{\bf Step 2.} From (\ref{B2.47}) and (\ref{F2.48}), we obtain (\ref{V2.45}). Thus the uniqueness follows.
\end{proof}

We consider a duality problem of system (\ref{2.40})
\be\left\{\begin{array}{l}\L_Y^* V=F\qfq x\in\Om_{-\varepsilon},\\
 \<V,\nu\>=p\qfq x\in\Ga_{-\varepsilon}, \quad \<V,QX\>=q\qfq x\in\Ga_{b_1}.\quad\end{array}\right.\label{2.45}\ee

By a duality of operators $\L_Y$ and $\L_Y^*$ a similar argument as for Theorem \ref{t2.1} yields the following.

\begin{thm}\label{t2.2} Let vector field $X$ be given in $(\ref{2.32})$ and $Y\in\CC^0(\overline{\Om}_{-\varepsilon}, T).$  Let $F\in\LL^2(\Om_{-\varepsilon},T),$  $p\in\LL^2(\Ga_{-\varepsilon}),$ and $q\in\LL^2(\Ga_{b_1})$ be given. Then problem $(\ref{2.45})$ admits a unique solution $V\in D(\L_Y^*)\cap\LL^2_{b_1,-\varepsilon}(\Om_{-\varepsilon},T),$ which satisfies
\beq(V,\L_YW)_{\LL^2(\Om_{-\varepsilon},T)}&&=(W,F)_{\LL^2(\Om_{-\varepsilon},T)}+\int_{\Ga_{-\varepsilon}}\pounds_1p\<W,\nu\>d\Ga\nonumber\\
&&\quad-\int_{\Ga_{b_1}}\pounds_2q\<W,QX\>d\Ga\label{2.41}\eeq for $W\in \CC^1_{-\varepsilon,b_1}(\Om_{-\varepsilon},T),$ where $\pounds_1$ and $\pounds_2$ are given in $(\ref{n2.17x}).$ Furthermore, the following estimate holds true.
\beq\|V\|^2_{\LL^2(\Om_{-\varepsilon},T)}&&+\|V\|^2_{\LL^2(\Ga_{-\varepsilon},T)}+\|V\|^2_{\LL^2(\Ga_{b_1},T)}\nonumber\\
&&\quad\leq C(\|F\|^2_{\LL^2(\Om_{-\varepsilon},T)}+\|p\|^2_{\LL^2(\Ga_{-\varepsilon})}+\|q\|^2_{\LL^2(\Ga_{b_1})}).\label{L2.51}\eeq
\end{thm}

\setcounter{equation}{0}
\def\theequation{3.\arabic{equation}}
\section{$\WW^{m,2}$ Solutions of the tensor equation of mixed type }
\hskip\parindent Consider problem
\be\left\{\begin{array}{l}\L_0 V=F\qfq x\in\Om_{-\varepsilon},\\
\<V,QX\>=p\qfq x\in\Ga_{-\varepsilon},\quad \<V,\nu\>=q\qfq x\in\Ga_{b_1}, \end{array}\right.\label{3.1}\ee where $X$ is given in (\ref{2.32}) and $\L_0$ is given in (\ref{2.10}) with $Y=0.$

In this section, we shall establish the following.

\begin{thm}\label{t3.1} Let $S$ be of $\CC^{m+3}.$ Let $F\in\WW^{m,2}(\Om_{-\varepsilon},T),$ $p\in\WW^{m,2}(\Ga_{-\varepsilon}),$ and $q\in\WW^{m,2}(\Ga_{b_1})$ be given. Then problem $(\ref{3.1})$ admits a unique solution
$V\in\WW^{m,2}(\Om_{-\varepsilon},T)$ satisfying
\be \|V\|^2_{\WW^{m,2}(\Om_{-\varepsilon},T)}\leq C( \|F\|^2_{\WW^{m,2}(\Om_{-\varepsilon},T)}+ \|q\|^2_{\WW^{m,2}(\Ga_{b_1})}+ \|p\|^2_{\WW^{m,2}(\Ga_{-\varepsilon})}).\label{V3.2}\ee
\end{thm}

By a similar argument and the duality, we have
\begin{thm}\label{t3.2} Let $S$ be of $\CC^{m+3}.$ Let $F\in\WW^{m,2}(\Om_{-\varepsilon},T),$ $p\in\WW^{m,2}(\Ga_{b_1}),$ and $q\in\WW^{m,2}(\Ga_{-\varepsilon})$ be given. Then problem
\be\left\{\begin{array}{l}\L_0^* V=F\qfq x\in\Om_{-\varepsilon},\\
\<V,\nu\>=p\qfq x\in\Ga_{-\varepsilon},\quad \<V,QX\>=q\qfq x\in\Ga_{b_1}, \end{array}\right.\label{4*.9}\ee admits a unique solution
$V\in\WW^{m,2}(\Om_{-\varepsilon},T)$ satisfying $(\ref{V3.2}).$
\end{thm}

\begin{thm}\label{tt4.1} Let $S$ be of $\CC^3.$ Let $p\in\WW^{1,2}(\Ga_{-\varepsilon})$ and $q\in\WW^{1,2}(\Ga_{b_1})$ be given. Suppose that $F\in\LL^2(\Om_{-\varepsilon},T)$ be given such that
$$F|_{\Sigma_{-\varepsilon}^{\varepsilon/2}}\in\WW^{1,2}(\Sigma_{-\varepsilon}^{\varepsilon/2},T),$$ where
$$\Sigma_{s_1}^{s_2}=\{\,\a(t,s)\,|\,(t,s)\in[0,a)\times(s_1,s_2)\,\}\qfq -b_0\leq s_1<s_2\leq b_1.$$
Then problem $(\ref{3.1})$ admits a unique solution $V\in\WW^{1,2}(\Om_{-\varepsilon},T)$ satisfying
\beq\|V\|^2_{\WW^{1,2}(\Om_{-\varepsilon},T)}&&\leq C(\|F\|^2_{\LL^2(\Om_{-\varepsilon},T)}+\|F\|^2_{\WW^{1,2}(\Sigma_{-\varepsilon}^{\varepsilon/2},T)}\nonumber\\
&&\quad+\|p\|^2_{\WW^{1,2}(\Ga_{-\varepsilon})}+\|q\|^2_{\WW^{1,2}(\Ga_{b_1})}).\label{4.101}\eeq If $S$ is of $\CC^{m+3},$ then the above $V$ satisfies
\beq\|V\|^2_{\WW^{m,2}(\Om_{-\varepsilon},T)}&&\leq C(\|F\|^2_{\WW^{m-1,2}(\Om_{-\varepsilon},T)}+\|F\|^2_{\WW^{m,2}(\Sigma_{-\varepsilon}^{\varepsilon/2},T)}\nonumber\\
&&\quad+\|p\|^2_{\WW^{m,2}(\Ga_{-\varepsilon})}+\|q\|^2_{\WW^{m,2}(\Ga_{b_1})}).\label{4.102}\eeq
\end{thm}

By the duality and a similar argument in the proof of Theorem \ref{tt4.1}, we have the following. The proof is omitted.

\begin{thm}\label{tt4.2} Let $S$ be of $\CC^3.$ Let $p\in\WW^{1,2}(\Ga_{-\varepsilon})$ and $q\in\WW^{1,2}(\Ga_{b_1})$ be given. Suppose that $F\in\LL^2(\Om_{-\varepsilon},T)$ be given such that
$$F|_{\Sigma_{-\varepsilon}^{\varepsilon/2}}\in\WW^{1,2}(\Sigma_{-\varepsilon}^{\varepsilon/2},T).$$
Then problem $(\ref{4*.9})$ admits a unique solution $V\in\WW^{1,2}(\Om_{-\varepsilon},T)$ satisfying $(\ref{4.101}).$
If $S$ is of $\CC^{m+3},$ then the above $V$ satisfies $(\ref{4.102}).$
\end{thm}

The proofs of Theorems \ref{t3.1} and \ref{tt4.1} will be given in the end of this section.

We  assume that $V\in\CC^1(\Om_{-\varepsilon},T)$ satisfies equation
\be \L_0V=F\qfq x\in\Om_{-\varepsilon}.\ee
For convenience, we denote
\be\left\{\begin{array}{l}\div_g Q\nabla\n V-\eta\<V,QX\>=f_1\qfq x\in\Om_{-\varepsilon},\\
\div_g V-\eta\<V,\nabla\n X\>=f_2\qfq x\in\Om_{-\varepsilon}.\end{array}\right.\label{f3.3}\ee Then
$$F=e^{-\g\kappa}(f_1QX+f_2\nabla\n X).$$

Let $\Phi\in TS$ be given such that
$$\<\nabla\n\Phi,\Phi\>\not=0\qfq x\in \overline{S}.$$  We define an operator $\RR_\Phi:$ $\LL^2(S,T)\rw\LL^2(S,T)$ by
\be \RR_\Phi V=[\Phi,Q\nabla\n V]-Q\nabla\n[\Phi,V]\qfq V\in TS.\label{2.94}\ee
\begin{lem} Set
\be h_1=\frac{\<\RR_\Phi\Phi,Q\Phi\>}{\<\nabla\n\Phi,\Phi\>},\quad
h_2=\frac{\<\RR_\Phi Q\Phi,Q\Phi\>
-h_1\<\nabla\n Q\Phi,\Phi\>}{|\Phi|^2},\label{h2.31}\ee
\be Z=\frac1{|\Phi|^2}(\RR_\Phi-h_1Q\nabla\n-h_2\id)^T\Phi.\label{2.96}\ee
Then
\be \RR_\Phi=h_1Q\nabla\n+h_2\id+Z\otimes\Phi.\label{P27}\ee
\end{lem}

\begin{proof} We have
\beq\<(\RR_\Phi-h_1Q\nabla\n-h_2\id)\Phi,Q\Phi\>&&=\<\RR_\Phi\Phi,Q\Phi\>-h_1\<Q\nabla\n\Phi,Q\Phi\>-h_2\<\Phi,Q\Phi\>\nonumber\\
&&=\<\RR_\Phi\Phi,Q\Phi\>-h_1\<\nabla\n\Phi,\Phi\>=0,\label{P2.27}\eeq
and
\beq&&\<(\RR_\Phi-h_1Q\nabla\n-h_2\id)Q\Phi,Q\Phi\>=\<\RR_\Phi Q\Phi,Q\Phi\>-h_1\<Q\nabla\n Q\Phi,Q\Phi\>-h_2\<Q\Phi,Q\Phi\>\nonumber\\
&&=\<\RR_\Phi Q\Phi,Q\Phi\>-h_1\<\nabla\n Q\Phi,\Phi\>-(\<\RR_\Phi Q\Phi,Q\Phi\>-h_1\<\nabla\n Q\Phi,\Phi\>)=0.\label{P2.28}\eeq
Since $\frac{Q\Phi}{|\Phi|},$ $\frac{\Phi}{|\Phi|}$ forms an orthonormal basis, it follows from (\ref{P2.27}) and (\ref{P2.28}) that
$$\<(\RR_\Phi-h_1Q\nabla\n-h_2\id)W, Q\Phi\>=0\qfq W\in TS.$$
Thus
$$(\RR_\Phi-h_1Q\nabla\n-h_2\id)W=\frac1{|\Phi|^2}\<(\RR_\Phi-h_1Q\nabla\n-h_2\id)W,\Phi\>\Phi=\<W,Z\>\Phi=(Z\otimes\Phi)W$$ for all $W\in TS,$ that is, (\ref{2.96}).
\end{proof}

\begin{lem}\label{l3.2} For given $V\in\CC^1(S,T),$ we have
\be \L_{Z}[\Phi,V]=F_\Phi,\label{Phi2.29}\ee where $Z$ is given in $(\ref{2.96})$ and
\be F_\Phi=e^{-\g\kappa}(
f_{1\Phi}QX+f_{2\Phi}\nabla\n X),\label{F3.11}\ee
$$f_{1\Phi}
=\Phi f_1-h_1f_1-h_2f_2+\<V,H_1\>,\quad f_{2\Phi}=\Phi f_2+\<V,H_2\>,$$
\beq H_1&&=\nabla\n QD\div_g\Phi+[\Phi(\eta)-h_1\eta] QX-h_2\eta\nabla\n X\nonumber\\
&&\quad+\eta D_\Phi(QX)+h_1\div_gQ\nabla\n-\div_g\RR_\Phi-D\Phi(Z,\cdot),\nonumber\eeq
$$\quad H_2
=\Phi(\eta)\nabla\n X-D\div_g\Phi+\eta D\Phi(\nabla\n X,\cdot).$$
\end{lem}

\begin{proof} From (\ref{2.94}), (\ref{div[X,Y]}), (\ref{2.5}),  (\ref{P27}), and (\ref{f3.3}), we have
\beq\div_gQ\nabla\n[\Phi, V]&&=\div_g[\Phi,Q\nabla\n V]-\div_g\RR_\Phi V\nonumber\\
&&=\Phi(\div_gQ\nabla\n V)-Q\nabla\n V(\div_g\Phi)-\<DV, \RR_\Phi\>-\<V,\div_g\RR_\Phi\>\nonumber\\
&&=\Phi(\div_gQ\nabla\n V)-\<Q\nabla\n V,D\div_g\Phi\>-\<DV, h_1Q\nabla\n+h_2\id+Z\otimes\Phi\>-\<V,\div_g\RR_\Phi\>\nonumber\\
&&=\Phi(\div_gQ\nabla\n V)+\<V, \nabla\n QD\div_g\Phi\>-h_1(\div_gQ\nabla\n V-\<V,\div_gQ\nabla\n\>)\nonumber\\
&&\quad-h_2\div_gV-\<D_\Phi V, Z\>-\<V,\div_g\RR_\Phi\>\nonumber\\
&&=\Phi f_1-h_1f_1-h_2f_2+\<[\Phi,V], \eta QX-Z\>-\<D_V\Phi,Z\>\nonumber\\
&&\quad+\<V,\,\nabla\n QD\div_g\Phi+[\Phi(\eta)-h_1\eta] QX-h_2\eta\nabla\n X\nonumber\\
&&\quad+\eta D_\Phi(QX)+h_1\div_gQ\nabla\n-\div_g\RR_\Phi\>.\nonumber\eeq
Thus
$$f_{1\Phi}=\div_gQ\nabla\n[\Phi, V]+\<[\Phi,V], Z-\eta QX\>=\Phi f_1-h_1f_1-h_2f_2+\<V,H_1\>.$$

Similarly, we compute
$$\div_g[\Phi,V]=\Phi\div_gV-V\div_g\Phi=\Phi f_2+\eta\<[\Phi,V],\nabla\n X\>+\<V,H_2\>.$$
\end{proof}

Consider an operator on $\LL^2(\Om_{-\varepsilon},T):$
\be \L_0V
=e^{-\g\kappa}[(\div_g Q\nabla\n V-\<V,\eta QX\>)QX+(\div_g V-\eta\<V,\nabla\n X\>)\nabla\n X],\label{L2.10}\ee
\beq D(\L_0)&&=\{\,V\in \LL^2(\Om_{-\varepsilon},T)\big|\,\div_g V\in \LL^2(\Om_{-\varepsilon}), \div_gQ\nabla\n V\in \LL^2(\Om_{-\varepsilon}),\nonumber\\
&&\quad\quad\<V,QX\>|_{\Ga_{-\varepsilon}}=\<V,\nu\>|_{\Ga_{b_1}}=0\,\}.\nonumber\eeq
By Theorem \ref{t2.1}, $\L_0^{-1}:$ $\LL^2(\Om_{-\varepsilon},T)\rw\LL^2(\Om_{-\varepsilon},T)$ is bounded. Similarly, Theorem \ref{t2.2} implies
that  ${\L_0^*}^{-1}:$ $\LL^2(\Om_{-\varepsilon},T)\rw\LL^2(\Om_{-\varepsilon},T)$ is bounded, where the domain of operator $\L_0^*$ is
$$ D(\L_0^*)=\{\,V,\,\,\L_0^*V\in \LL^2(\Om_{-\varepsilon},T)\big|\,\<V,QX\>|_{\Ga_{b_1}}=\<V,\nu\>|_{\Ga_{-\varepsilon}}=0\,\},$$
where $\L_0^*$ is given in (\ref{2.13}) with $Y=0.$

Let $X$ be given in (\ref{2.32}). Then $QX,$ $\nabla\n X$ forms a vector field basis on $\overline{\Om}_{-\varepsilon}.$ Denote its conjugate basis by $F_1,$ $F_2,$ where
\be F_1=\frac{Q\nabla\n X}{\<\nabla\n X,X\>},\quad F_2=\frac{X}{\<\nabla\n X,X\>}\qfq x\in\Om_{-\varepsilon}.\label{F3.12}\ee
We define operator $\K:$ $\LL^2(\Om_{-\varepsilon},T)\rw \LL^2(\Om_{-\varepsilon},T)$ by
\be
\K F=D_\Phi F+\M F+\N \L_0^{-1}(F),\label{K3.12}\ee where
$$\M=[D_\Phi F_1+(\g\<\Phi,D\kappa\>-h_1)F_1-h_2F_2]\otimes QX
+(D_\Phi F_2+\g\<\Phi,D\kappa\>F_2)\otimes \nabla\n X,$$
$$\N=e^{-\g\kappa}(H_1\otimes QX+H_2\otimes \nabla\n X).$$

\begin{lem} Let $\K$ be given in $(\ref{K3.12}).$ Then
\be
(\K F,G)_{\LL^2(\Om_{-\varepsilon},T)}=(F,\K^* G)_{\LL^2(\Om_{-\varepsilon},T)}+\int_{\partial\Om_{-\varepsilon}}\<F,G\>\<\Phi,\nu\>d\Ga,\label{K3.13}\ee where
\be\K^*G
=-D_\Phi G-(\div_g \Phi) G+\M^TG+{\L_0^*}^{-1}(\N^TG).\label{G3.15}\ee
\end{lem}

\begin{proof} We have
\beq\<\K F,G\>&&=\<D_\Phi F+\M F+\N \L_0^{-1}(F),\,G\>\nonumber\\
&&=\div_g(\<F,G\>\Phi)+\<F,-(\div_g\Phi)G-D_\Phi G\>+\<F,\M^TG\>+\<\L_0^{-1}F,\N^TG\>.\nonumber\eeq
We integrate the above identity over $\Om_{-\varepsilon}$ to have (\ref{K3.13}).
\end{proof}

\begin{lem}
We have
\be\L_Z[\Phi,V]=\K \L_0V\qfq V\in\WW^{1,2}(\Om_{-\varepsilon}),\label{L3.17}\ee where $Z$ is given in $(\ref{2.96}).$
\end{lem}

\begin{proof} For given $V\in\WW^{1,2}(\Om_{-\varepsilon}),$ set
$$\div_g Q\nabla\n V-\<V,\eta QX\>=f_1,\quad\div_g V-\eta\<V,\nabla\n X\>=f_2.$$
By Lemma \ref{l3.2}
$$\L_Z[\Phi,V]=F_\Phi,$$ where $F_\Phi$ is given in Lemma \ref{l3.2}.

 Next, we shall prove
$$\K\L_0V=F_\Phi.$$

Noting that
$$\L_0V=e^{-\g\kappa}(f_1QX+f_2\nabla\n X),$$ we have
$$D_\Phi\L_0V=-\g\<\Phi,D\kappa\>\L_0V+e^{-\g\kappa}(\Phi f_1QX+\Phi f_2\nabla\n X)+e^{-\g\kappa}[f_1D_\Phi(QX)+f_2D_\Phi(\nabla\n X)].$$
Since $F_1,$ $F_2$ is the conjugate basis of $QX,$ $\nabla\n X,$ it follows that
\beq\M\L_0V&&=\<D_\Phi F_1+(\g\<\Phi,D\kappa\>-h_1)F_1-h_2F_2,\,\L_0V\> QX
+\<D_\Phi F_2+\g\<\Phi,D\kappa\>F_2,\,\L_0V\> \nabla\n \nonumber\\
&&=e^{-\g\kappa}[-f_1D_\Phi(QX)-f_2D_\Phi(\nabla\n X)]+\g\<\Phi,D\kappa\>\L_0V-e^{-\g\kappa}(h_1f_1+h_2f_2)QX.\nonumber\eeq
Thus, by (\ref{F3.11}), we obtain
\beq\K\L_0V&&=D_\Phi\L_0V+\M\L_0V+\N \L_0^{-1}(\L_0V)\nonumber\\
&&=e^{-\g\kappa}\{[\Phi f_1-(h_1f_1+h_2f_2)+\<H_1,V\>]QX+(\Phi f_2+\<H_2,V\>)\nabla\n X\}=F_\Phi.\nonumber\eeq
\end{proof}

Now we construct a special $\Phi$ as follows.
Define a function $\psi$ on $S$ by
$$\psi(\a(t,s))=s\qfq x=\a(t,s)\in S.$$ Then
$$\<D\psi,\a_t\>=0,\quad\<D\psi,\a_s\>=1\qfq x\in S.$$
Let
$$\Phi_0=QD\psi\qfq x\in S.$$ Then
$$\<\Phi_0,\nu\>=\<QD\psi,\pm\frac{Q\a_t}{|\a_t|}\>=0\qfq x\in\pl\Om_{-\varepsilon}.$$ Set
$$\Phi=e^\var\Phi_0\qfq x\in\Om_{\varepsilon},$$ where $\var$ is a function given in Lemma \ref{l3.4} below. Since
\be D\psi=\frac{\<D\psi,Q\a_t\>}{|\a_t|^2}Q\a_t\qfq x\in\Om_{-\varepsilon}\label{psi3.14}\ee with $\<D\psi,Q\a_t\>\not=0,$ from (\ref{Pi2.80}), we have
$$\<\nabla\n\Phi,\Phi\>=e^{2\var}\<\nabla\n\Phi_0,\Phi_0\>=e^{2\var}\frac{\<D\psi,Q\a_t\>^2}{|\a_t|^4}\Pi(\a_t,\a_t)>0\qfq x\in\overline{\Om}_{-\varepsilon}.$$

\begin{lem}\label{l3.4} There exists a function $\var$ on $\overline{\Om}_{-\varepsilon}$ such that
\be \<D_\Phi D\psi+\M D\psi, QD\psi\>=0\qfq x\in\Om_{-\varepsilon}\setminus S^+,\label{M3.16}\ee
\be\<[\Phi,X],QX\>=0\qfq x\in\Ga_{b_1},\label{X3.17}\ee where $X$ is given in $(\ref{2.32}).$ Furthermore, for $W\in\CC^1(\Om_{-\varepsilon},T),$ we have
\be \<[\Phi,W],D\psi\>=\Phi\<W,D\psi\>\qfq x\in\Om_{-\varepsilon}\setminus S^+,\label{Phi3.18}\ee
\be\<[\Phi,W],QX\>=\Phi\<W,QX\>+\<W,QX\>\<[\Phi,\frac1{|X|^2}QX],QX\>\qfq x\in\Ga_{b_1}.\label{Phi3.19}\ee
\end{lem}

\begin{proof}

Noting that $\Phi=e^\var\Phi_0=e^\var QD\psi,$ we have
\beq\<D_\Phi D\psi+\M D\psi, QD\psi\>&&=e^\var\<D_{\Phi_0}D\psi,\Phi_0\>+\<D_\Phi F_2+\g\<\Phi,D\kappa\>F_2,D\psi\>\<QD\psi,\nabla\n X\>\nonumber\\
&&\quad+\<D_\Phi F_1+(\g\<\Phi,D\kappa\>-h_1)F_1-h_2F_2,\,D\psi\>\<D\psi, QX\>\nonumber\\
&&=e^\var h-(h_1\<F_1,D\psi\>+h_2\<F_2,\,D\psi\>)\<D\psi, QX\>,\label{psi3.19}\eeq where
\beq h&&=\<D_{\Phi_0}D\psi,\Phi_0\>+\<D_{\Phi_0} F_2+\g\<\Phi_0,D\kappa\>F_2,D\psi\>\<QD\psi,\nabla\n X\>\nonumber\\
&&\quad+\<D_{\Phi_0} F_1+\g\<\Phi_0,D\kappa\>F_1, D\psi\>\<D\psi,QX\>.\nonumber\eeq

On the other hand, from (\ref{2.94}), we have
\beq\<\RR_\Phi \Phi,Q\Phi\>&&=\<D_\Phi(Q\nabla\n \Phi)-D_{Q\nabla\n \Phi}\Phi,\,Q\Phi\>\nonumber\\
&&=e^{3\var}[\<\RR_{\Phi_0} \Phi_0,Q\Phi_0\>+\Phi_0(\var)\<\nabla\n\Phi_0,\Phi_0\>].\nonumber\eeq
Similarly, we obtain
$$ \<\RR_\Phi Q\Phi,Q\Phi\>=e^{3\var}[\<\RR_{\Phi_0}Q\Phi_0,Q\Phi_0\>+(Q\Phi_0)(\var)\<\nabla\n \Phi_0,\Phi_0\>].$$
Thus
$$h_1=e^\var[\frac{\<\RR_{\Phi_0}\Phi_0,Q\Phi_0\>}{\<\nabla\n\Phi_0,\Phi_0\>}+\Phi_0(\var)],$$
\beq h_2&&=\frac{e^{\var}}{|\Phi_0|^2}\{\<\RR_{\Phi_0}Q\Phi_0,Q\Phi_0\>-\frac{\<\RR_{\Phi_0}\Phi_0,Q\Phi_0\>}{\<\nabla\n\Phi_0,\Phi_0\>}\<\nabla\n Q\Phi_0,\Phi_0\>\nonumber\\
&&\quad-(Q\Phi_0)(\var)\<\nabla\n \Phi_0,\Phi_0\>+\Phi_0(\var)\<\nabla\n Q\Phi_0,\Phi_0\>\}.\nonumber\eeq
We then obtain
\beq h_1\<F_1,D\psi\>+h_2\<F_2,\,D\psi\>&&=e^\var\{\hat{h}+\Phi_0(\var)(\<F_1,D\psi\>+\frac{\<\nabla\n Q\Phi_0,\Phi_0\>}{|\Phi_0|^2}\<F_2,D\psi\>)\nonumber\\
&&\quad-(Q\Phi_0)(\var)\frac{\<\nabla\n \Phi_0,\Phi_0\>}{|\Phi_0|^2}\<F_2,D\psi\>\},\nonumber\eeq
where
\beq\hat{h}&&=\frac{\<\RR_{\Phi_0}\Phi_0,Q\Phi_0\>}{\<\nabla\n\Phi_0,\Phi_0\>}\<F_1,D\psi\>+\frac{\<F_2,D\psi\>}{|\Phi_0|^2}[\<\RR_{\Phi_0}Q\Phi_0,Q\Phi_0\>-\frac{\<\RR_{\Phi_0}\Phi_0,Q\Phi_0\>}{\<\nabla\n\Phi_0,\Phi_0\>}\<\nabla\n Q\Phi_0,\Phi_0\>].\nonumber\eeq
It follows from (\ref{psi3.19}) that
\be\<D_\Phi D\psi+\M D\psi, QD\psi\>=e^\var[h-\hat h+H(\var)],\label{H3.21}\ee where
$$H=(\<F_1,D\psi\>-\frac{\<\nabla\n D\psi,QD\psi\>}{|D\psi|^2}\<F_2,D\psi\>)QD\psi+\frac{\<\nabla\n QD\psi,QD\psi\>\<F_2,D\psi\>}{|D\psi|^2}D\psi.$$

By (\ref{nu2.31}) and (\ref{psi3.14}),
$$\<QD\psi,\nu\>=0\qfq x\in\Ga_{-\varepsilon}.$$
From (\ref{psi3.14}), (\ref{F3.12}), and (\ref{X2.45}), we have
$$\<H,\nu\>=\frac{\<D\psi,Q\a_t\>^2}{|\a_t|^4\Pi(X,X)}\Pi(\a_t,\a_t)\<X,\nu\><0\qfq x\in\Ga_{-\varepsilon}.$$
Thus, when given $\varepsilon>0$ is small, the integral curves of the vector field $H$ initiating from $\Ga_{-\varepsilon}$ direct to the interior of $\Om_{-\varepsilon}$ can across $\Ga_0.$ We solve equation (\ref{H3.21}) along those integral curves to have $\var$ to satisfy (\ref{M3.16}).

By a similar argument, we obtain a function $\var$ in a neighborhood $\Ga_{b_1}$ such that (\ref{X3.17}) holds.

Next, we prove (\ref{Phi3.18}) and (\ref{Phi3.19}). By the symmetry of $D^2\psi, $ we have
$$\<[\Phi,W],D\psi\>=\Phi\<W,D\psi\>-\<W,D_\Phi D\psi\>-W\<\Phi,D\psi\>+\<\Phi,D_WD\psi\>=\Phi\<W,D\psi\>.$$ Moreover, since $\frac{QX}{|X|},$ $\frac{X}{|X|}$ forms an orthonormal frame, it follows from
(\ref{X3.17}) that
\beq\<[\Phi,W],QX\>&&=\<[\Phi,\frac{\<W,X\>}{|X|^2}X],QX\>+\<[\Phi,\frac{\<W,QX\>}{|X|^2}QX],QX\>\nonumber\\
&&=\frac{\<W,X\>}{|X|^2}\<[\Phi,X],QX\>+\Phi\<W,QX\>+\<W,QX\>\<[\Phi,\frac1{|X|^2}QX],QX\>\nonumber\\
&&=\Phi\<W,QX\>+\<W,QX\>\<[\Phi,\frac1{|X|^2}QX],QX\>\qfq x\in\Ga_{b_1}.\nonumber\eeq
\end{proof}

\begin{lem}\label{l3.6} Let $k>0$ be given large and $W\in\CC^1(\Om_{-\varepsilon},T).$ Then
$$\<\K^*[\Phi,W]+kW,\nu\>|_{\Ga_{-\varepsilon}}=0\quad\Rightarrow\quad \<W,\nu\>|_{\Ga_{-\varepsilon}}=\<[\Phi,W],\nu\>|_{\Ga_{-\varepsilon}}=0,$$
$$\<\K^* [\Phi,W]+kW,QX\>|_{\Ga_{b_1}}=0\quad\Rightarrow\quad  \<W,QX\>\big|_{\Ga_{b_1}}=\<[\Phi,W],QX\>\big|_{\Ga_{b_1}}=0,$$ where operator $\K^*$ is given in $(\ref{G3.15}).$
\end{lem}

\begin{proof} Since $\frac{QD\psi}{|D\psi|},$  $\frac{D\psi}{|D\psi|}$ forms an orhonormal frame, from (\ref{M3.16}), we have
$$D_\Phi D\psi+\M D\psi=\frac{\<D_\Phi D\psi+\M D\psi,D\psi\>}{|D\psi|^2}D\psi.$$
Noting that $D\psi|_{\Ga_{-\varepsilon}}=\pm\frac{\<D\psi,Q\a_t\>}{|\a_t|}\nu,$ we have
$$\<{\L_0^*}^{-1}(\N^T[\Phi,W]),D\psi\>|_{\Ga_{-\varepsilon}}=0.$$ It follows from (\ref{Phi3.18}) that
\beq\<\K^*[\Phi,W],D\psi\>&&=\<-D_\Phi[\Phi,W]-(\div_g \Phi)[\Phi,W]+\M^T[\Phi,W],\,D\psi\>\nonumber\\
&&=-\Phi^2\<W,D\psi\>-(\div_g \Phi)\<W,D\psi\>+\<[\Phi,W], D_\Phi D\psi+\M D\psi\>\nonumber\\
&&=-\Phi^2\<W,D\psi\>-[\div_g \Phi-\frac{\<D_\Phi D\psi+\M D\psi,D\psi\>}{|D\psi|^2}]\<W,D\psi\>.\nonumber\eeq
Thus we obtain
$$\<\K^*[\Phi,W]+kW,D\psi\>\<W,D\psi\>\geq-\Phi(\<W,D\psi\>\Phi\<W,D\psi\>)+(\Phi\<W,D\psi\>)^2+(k-C)\<W,D\psi\>^2.$$ We integrate the above inequality over $\Ga_{-\varepsilon}$ to have
$$\|\<\K^*[\Phi,W]+kW,D\psi\>\<W,D\psi\>\|^2_{\LL^2(\Ga_{-\varepsilon})}\geq\si(\|\Phi\<W,D\psi\>\|^2_{\LL^2(\Ga_{-\varepsilon})}+\|\<W,D\psi\>\|^2_{\LL^2(\Ga_{-\varepsilon})}$$ for given $k>0$ large, that is what we want.

A similar computation yields the second conclusion.
\end{proof}

We introduce a Hilbert space by
$$\V_\Phi=\{\,W,\,\,[\Phi,W]\in\LL^2(\Om,T)\,|\,\<W,\nu\>|_{\Ga_{b_1}}\in\WW^{1,2}(\Ga_{b_1}),\,\,\<W,QX\>|_{\Ga_{-\varepsilon}}\in\WW^{1,2}(\Ga_{-\varepsilon})\,\},$$
$$\|W\|^2_{\V_\Phi}=\|[\Phi,W]\|^2_{\LL^2(\Om,T)}+\|W\|^2_{\LL(\Om,T)}+\|\<W,QX\>\|^2_{\WW^{1,2}(\Ga_{\varepsilon})}+\|\<W,\nu\>\|^2_{\WW^{1,2}(\Ga_{b_1})}.$$

\begin{lem}\label{l3.7} Let $W_0\in\CC^1(\Om_{-\varepsilon},T)$ be given with $\<W_0,\nu\>|_{\Ga_{-\varepsilon}}=0$ and $\<W_0,QX\>|_{\Ga_{b_1}}=0.$ Then there exists a unique $W\in\V_\Phi$ such that
$$\K^* [\Phi,W]+kW=W_0\qfq x\in\Om_{-\varepsilon}$$ for given $k>0$ large.
\end{lem}

\begin{proof}

Define a bilinear form $\a:$ $\V_{\Phi}\times\WW^{1,2}_0(\Om,T)\rw\R$ by
$$\a(W,V)=(\K^* [\Phi,W]+kW,V)_{\LL^2(\Om_{-\varepsilon})}.$$ Since $\<\Phi,\nu\>|_{\pl\Om}=0,$ from (\ref{K3.13}), we have
\beq\a(W,V)&&=([\Phi,W],\K V)_{\LL^2(\Om_{-\varepsilon})}+k(W,V)_{\LL^2(\Om_{-\varepsilon})}\nonumber\eeq which yield
$$|\a(W,V)|\leq C\|W\|_{\V_\Phi}\|V\|_{\WW^{1,2}(\Om,T)}\qfq W\in\V_\Phi,\quad W\in\WW^{1,2}_0(\Om,T).$$
By (\ref{K3.12}), we have
\beq\a(W,W)&&=([\Phi,W],\K W)_{\LL^2(\Om_{-\varepsilon})}+k\|W\|^2_{\LL^2(\Om_{-\varepsilon})}\nonumber\\
&&=([\Phi,W],D_\Phi W+\M W+\N \L_0^{-1}(W))_{\LL^2(\Om_{-\varepsilon})}+k\|W\|^2_{\LL^2(\Om_{-\varepsilon})}\nonumber\\
&&=([\Phi,W],[\Phi,W]+D_W\Phi+\M W+\N \L_0^{-1}(W))_{\LL^2(\Om_{-\varepsilon})}+k\|W\|^2_{\LL^2(\Om_{-\varepsilon})}\nonumber\\
&&\geq\|[\Phi,W]\|^2_{\LL^2(\Om_{-\varepsilon})}+(k-C)\|W\|^2_{\LL^2(\Om_{-\varepsilon})}\label{a3.23}\\
&&\geq\|W\|^2_{\V_\Phi}\qfq W\in\WW_0^{1,2}(\Om_{-\varepsilon},T).\nonumber\eeq

Let $\F(V)=(W_0,V)$ for $V\in\V_\Phi.$ Then  $\F$ is bounded on $\V_\Phi.$ By Theorem A given in Appendix there exists $W\in\V_\Phi$ such that
$$\F(V)=\a(W,V)\qfq V\in\WW_0^{1,2}(\Om_{-\varepsilon},T),$$ which implies
$$\K^* [\Phi,W]+kW=W_0\qfq x\in\Om_{-\varepsilon}.$$
Moreover, a similar computation as in (\ref{a3.23}) gives
$$(\K^* [\Phi,W]+kW,W)_{\LL^2(\Om_{-\varepsilon})}\geq\|[\Phi,W]\|^2_{\LL^2(\Om_{-\varepsilon})}+(k-C)\|W\|^2_{\LL^2(\Om_{-\varepsilon})}.$$ Thus the uniqueness follows.
\end{proof}

\begin{pro}\label{p3.1} Let $F\in\LL^2(\Om_{-\varepsilon},T)$ be given with $[\Phi,F]\in\LL^2(\Om_{-\varepsilon},T).$
Let  $p\in\WW^{1,2}(\Ga_{-\varepsilon})$ and $q\in\WW^{1,2}(\Ga_{b_1})$ be given. Then problem $(\ref{3.1})$ admits a unique
solution $V\in\V_\Phi$ satisfying
\be\|V\|^2_\Phi\leq C(\|[\Phi,F]\|^2_{\LL^2(\Om_{-\varepsilon})}+\|F\|^2_{\LL^2(\Om_{-\varepsilon})}+\|p\|^2_{\WW^{1,2}(\Ga_{-\varepsilon})}+\|q\|^2_{\WW^{1,2}(\Ga_{b_1})}).\label{F3.26}\ee
\end{pro}

\begin{proof} Let $k>0$ be given large.
Let
\be\D_0=\{\,W\in\CC^1(\Om_{-\varepsilon},T)\,|\,\<W,\nu\>|_{\Ga_{-\varepsilon}}=\<W,QX\>|_{\Ga_{b_1}}=0\,\},\label{D3.27}\ee
$$\W=\{\,W\in\V_\Phi\,|\,\mbox{there is $W_0\in\D_0$ such that $\K^* [\Phi,W]+kW=W_0$}\,\}.$$
By Lemma \ref{l3.7}, we introduce an inner product on $\W$ by
$$(V,W)_\W=(V_0,W_0)_{\WW^{1,2}(\Om_{-\varepsilon},T)},$$ where $\K^* [\Phi,V]+kV=V_0$ and $\K^* [\Phi,W]+kW=W_0.$ Then $\W\subset\V_\Phi$ is an inner product space. It follows from Lemma \ref{l3.6} that
\be\<W,QX\>|_{\Ga_{b_1}}=\<[\Phi,W],QX\>\big|_{\Ga_{b_1}}=\<W,\nu\>|_{\Ga_{-\varepsilon}}=\<[\Phi,W],\nu\>|_{\Ga_{-\varepsilon}}=0\label{W3.25}\ee for $W\in\W.$

{\bf Step 1.}\,\,\,We define a bilinear form $\a:$ $\V_\Phi\times\W\rw\R$ by
$$\a(V,W)=-(V,\L_0^*(\K^* [\Phi,W]+kW))_{\LL^2(\Om_{-\varepsilon})}.$$ Clearly, we have
$$|\a(V,W)|\leq C\|V\|_{\Phi}\|W\|_\W\qfq V\in\V_\Phi,\quad W\in\W.$$

Next, we prove the  coerciveness. From (\ref{2.14}), (\ref{W3.25}), (\ref{K3.13}), (\ref{L3.17}), (\ref{2.14}) (again), we compute
\beq&&-2(W,\L_0^*(\K^* [\Phi,W]+kW))_{\LL^2(\Om_{-\varepsilon})}=-2(\L_0W,\K^*[\Phi,W]+kW)_{\LL^2(\Om_{-\varepsilon})}\nonumber\\
&&\quad+2\int_{\Ga_{b_1}}\pounds_1\<W,\nu\>\<\K^* [\Phi,W]+kW,\nu\>d\Ga\nonumber\\
&&\quad-2\int_{\Ga_{-\varepsilon}}\pounds_2\<W,QX\>\<\K^* [\Phi,W]+kW,QX\>d\Ga\nonumber\\
&&=-2(\K\L_0[\Phi,W],[\Phi,W])_{\LL^2(\Om_{-\varepsilon})}-2k(\L_0W,W)_{\LL^2(\Om_{-\varepsilon},T)}\nonumber\\
&&\quad+2\int_{\Ga_{b_1}}\pounds_1\<W,\nu\>\<\K^* [\Phi,W]+kW,\nu\>d\Ga\nonumber\\
&&\quad-2\int_{\Ga_{-\varepsilon}}\pounds_2\<W,QX\>\<\K^* [\Phi,W]+kW,QX\>d\Ga\nonumber\\
&&=-([\Phi,W],\L_Z[\Phi,W]+\L_Z^*[\Phi,W])_{\LL^2(\Om_{-\varepsilon})}-k(W,\L_0W+\L_0^*W)_{\LL^2(\Om_{-\varepsilon},T)}\nonumber\\
&&\quad+\Ga_{b_1}(W)+\Ga_{-\varepsilon}(W),\label{L3.27}\eeq where
$$\Ga_{b_1}(W)=\int_{\Ga_{b_1}}|\pounds_1|[2\<W,\nu\>\<\K^* [\Phi,W]+kW,\nu\>-\<[\Phi,W],\nu\>^2-k\<W,\nu\>^2]d\Ga,$$
$$\Ga_{-\varepsilon}(W)=\int_{\Ga_{-\varepsilon}}|\pounds_2|[2\<W,QX\>\<\K^* [\Phi,W]+kW,QX\>-\<[\Phi,W],QX\>^2-k\<W,QX\>^2]d\Ga.$$

We estimate $\Ga_{-\varepsilon}(W)$ and $\Ga_{b_1}(W)$ below. By (\ref{X2.45}), $QX,$ $\nu$ forms a vector field basis along $\Ga_{-\varepsilon}.$ It follows from (\ref{W3.25}) that
$$W=\O(\<W,QX\>),\quad [\Phi,W]=\O(\<[\Phi,W],QX\>)\qfq x\in\Ga_{-\varepsilon}.$$
Set $\tau=\Phi/|\Phi|.$   Then
\beq&&2|\pounds_2|\<W,QX\>\<-D_\Phi [\Phi,W],QX\>\nonumber\\
&&=-2\tau\<[\Phi,W],|\pounds_2||\Phi|\<W,QX\>QX\>+2\<[\Phi,W],D_\tau(|\pounds_2||\Phi|\<W,QX\>QX)\>\nonumber\\
&&=-2\tau\<[\Phi,W],|\pounds_2||\Phi|\<W,QX\>QX\>+2|\pounds_2|\Phi\<W,QX\>\<[\Phi,W],QX\>\>\nonumber\\
&&\quad+2\<W,QX\>\<[\Phi,W],D_\tau(|\pounds_2||\Phi|QX)\>\nonumber\\
&&=\tau(\t{h})+2|\pounds_2|\<[\Phi,W],QX\>^2+\O(\<W,QX\>\<[\Phi,W],QX\>)\qfq x\in\Ga_{-\varepsilon},\nonumber\eeq where
$$\t{h}=-2\<[\Phi,W],|\pounds_2||\Phi|\<W,QX\>QX\>.$$
Noting that
$$\<{\L_0^*}^{-1}(\N^T[\Phi,W]),QX\>=0\qfq x\in\Ga_{-\varepsilon},$$ from (\ref{G3.15}), we obatin
\beq&&|\pounds_2|[2\<W,QX\>\<\K^* [\Phi,W]+kW,QX\>-\<[\Phi,W],QX\>^2-k\<W,QX\>^2]\nonumber\\
&&=2|\pounds_2|\<W,QX\>\<-D_\Phi [\Phi,W]-(\div_g \Phi)[\Phi,W]+\M^T[\Phi,W]+{\L_0^*}^{-1}(\N^T[\Phi,W]),QX\>\nonumber\\
&&\quad+k|\pounds_2|\<W,QX\>^2-|\pounds_2|\<[\Phi,W],QX\>^2\nonumber\\
&&=\tau(\t{h})+|\pounds_2|\<[\Phi,W],QX\>^2+k|\pounds_2|\<W,QX\>^2+\O(\<W,QX\>\<[\Phi,W],QX\>)\nonumber\\
&&\geq \tau(\t{h})+\frac{|\pounds_2|}2\<[\Phi,W],QX\>^2+(k-C)\<W,QX\>^2\nonumber\\
&&\geq \tau(\t{h})+\si(\Phi\<W,QX\>)^2+(k-C)\<W,QX\>^2\qfq x\in\Ga_{-\varepsilon}.\nonumber\eeq
Thus
$$\Ga_{-\varepsilon}(W)\geq\si\|\<W,QX\>\|^2_{\WW^{1,2}(\Ga_{-\varepsilon})}\qfq W\in\W$$ for given $k>0$ large. By (\ref{psi3.14}) a similar argument yields
$$\Ga_{b_1}(W)\geq \si\|\<W,\nu\>\|^2_{\WW^{1,2}(\Ga_{b_1})}\qfq W\in\W$$ for given $k>0$ large.

Using (\ref{2.34}) in (\ref{L3.27}), we obtain
$$\a(W,W)\geq\si\|W\|^2_{\V_\Phi}\qfq W\in\W.$$

{\bf Step 2.}\,\,\,Let $F\in\LL^2(\Om_{-\varepsilon},T)$ be given with $[\Phi,F]\in\LL^2(\Om_{-\varepsilon},T).$ Set
\beq\F(W)&&=-(F,\K^* [\Phi,W]+kW)_{\LL^2(\Om_{-\varepsilon},T)}
+\int_{\Ga_{-\varepsilon}}|\pounds_2| p\<\K^* [\Phi,W]+kW,QX\>d\Ga
\nonumber\\
&&+\int_{\Ga_{b_1}}|\pounds_1|q\<\K^* [\Phi,W]+kW,\nu\>d\Ga\qfq W\in\V_\Phi.\nonumber\eeq
Since $(F,\K^* [\Phi,W])_{\LL^2(\Om_{-\varepsilon},T)}=(\K F,[\Phi,W])_{\LL^2(\Om_{-\varepsilon},T)},$ it is easy to check that $\F$ is bounded on $\V_\Phi.$  By Theorem A given in Appendix there exists a $V\in\V_\Phi$ such that
\be\F(W)=\a(V,W)\qfq W\in\W,\label{F3.28}\ee
\be\|V\|_\Phi\leq C\|\F\|_{\V^*_\Phi}.\label{V3.31}\ee

{\bf Step 3.}\,\,\,We shall show that $V\in\V_\Phi$ given in (\ref{F3.28}) is the solution to problem (\ref{3.1}). Let $W_0\in\CC^1(\Om_{-\varepsilon},T)$ be given with $\<W_0,\nu\>|_{\Ga_{-\varepsilon}}=\<W_0,QX\>|_{\Ga_{b_1}}=0.$ By Lemma \ref{l3.7}, there is $W\in\V_\Phi$ such that $\K^*[\Phi,W]+kW=W_0.$ Then (\ref{F3.28}) becomes
\beq(V,\L_0^*W_0)_{\LL^2(\Om_{-\varepsilon})}&&=(F,W_0)_{\LL^2(\Om_{-\varepsilon},T)}-\int_{\Ga_{-\varepsilon}}|\pounds_2| p\<W_0,QX\>d\Ga
\nonumber\\
&&-\int_{\Ga_{b_1}}|\pounds_1|q\<W_0,\nu\>d\Ga\qfq W_0\in\D_0,\label{L3.31}\eeq where $\D_0$ is given in (\ref{D3.27}). Using (\ref{2.14}) in (\ref{L3.31}), one has
\beq &&(W_0,\L_0V)_{\LL^2(\Om,T)}=(F,W_0)_{\LL^2(\Om,T)}\nonumber\\
&&\quad+\int_{\Ga_{b_1}}|\pounds_1|(q-\<V,\nu\>)\<W_0,\nu\>d\Ga+\int_{\Ga_{-\varepsilon}}|\pounds_2|(p-\<V,QX\>)\<W_0,QX\>d\Ga\eeq for all $W_0\in\D_0.$ Thus $V\in\V_\Phi$ solves problem (\ref{3.1}).
Finally, (\ref{F3.26}) follows from (\ref{V3.31}).
\end{proof}

\begin{lem} Let $Y\in TM$ be with $|Y|=1.$ Then
\be \ii_YDZ=D\<Z,Y\>-\<Z,QY\>[QY,Y]\qfq Z\in TM.\label{4.1}\ee In addition,
\be-\id\otimes Q\nabla\n Z+Q\nabla\n\otimes Z=QZ\otimes Q\nabla\n Q-\nabla\n Z\otimes Q\qfq Z\in TM.\label{4.2}\ee
\end{lem}

\begin{proof} $|Y|=1$ implies that $QY,$ $Y$ is an orthonormal frame. Then
$$\ii_ZDY=\<D_ZY,QY\>QY=\<\ii_{QY}DY,Z\>QY\qfq Z\in TM.$$ Since
$$\<D_YQY,QY\>=\<D_{QY}Y,Y\>=0,$$ we have
\beq D\<Z,Y\>&&=QY\<Z,Y\>QY+Y\<Z,Y\>Y\nonumber\\
&&=(\<\ii_YDZ,QY\>+\<D_{QY}Y,QY\>\<Z,QY\>)QY+(\<\ii_YDZ,Y\>+\<D_YY,QY\>\<Z,QY\>)Y\nonumber\\
&&=\ii_ZDY+\<Z,QY\>(\<D_{QY}Y-D_YQY,QY\>QY+\<D_{QY}Y-D_YQY,Y\>Y)\nonumber\\
&&=\ii_ZDY+\<Z,QY\>[QY,Y],\nonumber\eeq that is (\ref{4.1}).

Let $E_1,$ $E_2$ be an orthonormal frame. Leting $v=E_i,$ $w=QE_j,$ and $P=\nabla\n$ in (\ref{2.7}) yields
$$\<E_i,E_j\>Q\nabla\n+\<\nabla\n E_i,QE_j\>\id=QE_i\otimes Q\nabla\n QE_j+\nabla\n E_i\otimes QE_j. $$ Thus
$$\<E_i,E_j\>Q\nabla\n Z+\<\nabla\n E_i,QE_j\>Z=\<Z,QE_i\> Q\nabla\n QE_j+\<Z,\nabla\n E_i\> QE_j. $$
Then
$$\<E_i,E_j\>\<Q\nabla\n Z,E_k\>-\<Q\nabla\n E_i,E_j\>\<Z,E_k\>=-\<QZ,E_i\>\< Q\nabla\n QE_j,E_k\>+\<\nabla\n Z, E_i\> \<QE_j,E_k\> $$
for all $1\leq i,$ $j,$ $k\leq2,$ that is (\ref{4.2}).
\end{proof}

\begin{lem}
 Let $\Om\subset\Om_{-\varepsilon}$ be a subregion and let $\theta\in\CC^1(\Om).$ Let $E\in T\Om$ be given with $|E|=1.$ Let $\si>0$ be given small.
Then
\beq&&\int_\Om\theta^2[Q^*\Pi(\ii_EDV,\ii_EDV)-\si|DV|^2]dg\nonumber\\
&&\leq C\int_\Om\theta^2[(\div_gQ\nabla\n V)^2+(\div_gV)^2]dg+C\|V\|^2_{\LL^2(\Om,T)}\nonumber\\
&&\quad-\int_{\pl\Om}\theta^2\<\nabla\n QE,QE\>[\frac12|V|^2\<\nu\wedge[QE,E],\E\>+\<V,QE\>\<\nu\wedge d\<V,E\>,\,\E\>]d\Ga\label{4.3}\eeq for all $V\in\WW^{1,2}(\Om,T),$ where $\nu$ is the outside normal of $\Om$ and $\E$ is the volume element of $M.$
\end{lem}

\begin{proof} Let $E_1,$ $E_2$ be an orthonormal frame. For given $W\in TM,$ using (\ref{4.2}), we have
\beq
&&\<\<DV,Q\nabla\n\>QE-\<DV,\id\>Q\nabla\n QE,W\>=\<DV,Q\nabla\n\>\<W,QE\>-\<DV,\id\>\<Q\nabla\n QE,W\>\nonumber\\
&&=\<DV\otimes W,Q\nabla\n\otimes QE\>-\<DV\otimes W,\id\otimes Q\nabla\n QE\>\nonumber\\
&&=\<DV\otimes W,\,-E\otimes Q\nabla\n Q-\nabla\n QE\otimes Q\>\nonumber\\
&&=-\sum DV(E_i,E_j)\<W,E_k\>(\<E,E_i\>\<Q\nabla\n QE_j,E_k\>+\<\nabla\n QE,E_i\>\<QE_j,E_k\>)\nonumber\\
&&=\<\nabla\n Q\ii_EDV,QW\>+\<\ii_{\nabla\n QE}DV,QW\>.\nonumber\eeq
Letting $W=\ii_EDV$ in the above identity gives
\beq&&\<\<DV,Q\nabla\n\>QE-\<DV,\id\>Q\nabla\n QE,\ii_EDV\>\nonumber\\
&&=Q^*\Pi(\ii_EDV,\ii_EDV)+\<\ii_{\nabla\n QE}DV,Q\ii_EDV\>.\label{4.4}\eeq
Since $QE,$ $E$ forms an orthonormal frame, we obtain
\beq&&\<\ii_{\nabla\n QE}DV,Q\ii_EDV\>=\<\nabla\n QE,QE\>\<\ii_{ QE}DV,Q\ii_EDV\>\nonumber\\
&&=\<\nabla\n QE,QE\>\<DV(QE,E)E+DV(QE,QE)QE,\,DV(E,E)QE-DV(E,QE)E\>\nonumber\\
&&=\<\nabla\n QE,QE\>[DV(QE,QE)DV(E,E)-DV(QE,E)DV(E,QE)]\nonumber\\
&&=\<\nabla\n QE,QE\>\<\ii_{QE}DV\wedge\ii_EDV,\,\E\>.\label{4.5}\eeq
On the other hand, it follows from (\ref{4.1}) that
\beq\ii_EDV\wedge\ii_{QE}DV&&=(D\<V,E\>-\<V,QE\>[QE,E])\wedge(D\<V,QE\>+\<V,E\>[QE,E])\nonumber\\
&&=D\<V,E\>\wedge D\<V,QE\>+(\<V,E\>D\<V,E\>+\<V,QE\>D\<V,QE\>)\wedge[QE,E]\nonumber\\
&&=d\<V,E\>\wedge d\<V,QE\>+\frac12d(|V|^2)\wedge[QE,E].\label{4.6}\eeq

Moreover, noting that $d^2=0,$  we compute
\beq&&d (\theta^2\<\nabla\n QE,QE\>\<V,QE\>d\<V,E\>)+\<V,QE\>d\<V,E\>\wedge d(\theta^2\<\nabla\n QE,QE\>)\nonumber\\
&&=d (\theta^2\<\nabla\n QE,QE\>\<V,QE\>)\wedge d\<V,E\>+\<V,QE\>d\<V,E\>\wedge d(\theta^2\<\nabla\n QE,QE\>)\nonumber\\
&&=[\<V,QE\>d (\theta^2\<\nabla\n QE,QE\>)+\theta^2\<\nabla\n QE,QE\>d\<V,QE\>]\wedge d\<V,E\>\nonumber\\
&&\quad+\<V,QE\>d\<V,E\>\wedge d(\theta^2\<\nabla\n QE,QE\>)\nonumber\\
&&=\theta^2\<\nabla\n QE,QE\>d\<V,QE\>\wedge d\<V,E\>,\label{4.7}\eeq and
\beq&&\frac12 d(\theta^2\<\nabla\n QE,QE\>|V|^2[QE,E])-\frac12|V|^2 d(\theta^2\<\nabla\n QE,QE\>[QE,E])\nonumber\\
&&=\frac12|V|^2 d(\theta^2\<\nabla\n QE,QE\>[QE,E])+\frac12d(|V|^2)\wedge(\theta^2\<\nabla\n QE,QE\>[QE,E])\nonumber\\
&&\quad-\frac12|V|^2 d(\theta^2\<\nabla\n QE,QE\>[QE,E])\nonumber\\
&&=\frac12d(|V|^2)\wedge(\theta^2\<\nabla\n QE,QE\>[QE,E]).\label{4.8}\eeq

Since
$$\div_gQ\nabla\n V=\<DV,Q\nabla\n\>,\quad\div_gV=\<DV,\id\>,$$ inserting (\ref{4.8}), (\ref{4.7}), (\ref{4.6}), and (\ref{4.5}) into (\ref{4.4}),  we obtain
\beq&&\theta^2\<(\div_gQPV)QE-(\div_gV)QPQE,\ii_EDV\>\nonumber\\
&&=\theta^2\<\<DV,Q\nabla\n\>QE-\<DV,\id\>QPQE,\ii_EDV\>\nonumber\\
&&=\theta^2Q^*\Pi(\ii_EDV,\ii_EDV)+\theta^2\<\ii_{\nabla\n QE}DV,QW\>\nonumber\\
&&=\theta^2Q^*\Pi(\ii_EDV,\ii_EDV)+\theta^2\<\nabla\n QE,QE\>\<\ii_{QE}DV\wedge\ii_EDV,\,\E\>\nonumber\\
&&=\theta^2Q^*\Pi(\ii_EDV,\ii_EDV)+\theta^2\<\nabla\n QE,QE\>\<d\<V,E\>\wedge d\<V,QE\>+\frac12d(|V|^2)\wedge[QE,E],\,\E\>\nonumber\\
&&=\theta^2Q^*\Pi(\ii_EDV,\ii_EDV)+\<d (\theta^2\<\nabla\n QE,QE\>\<V,QE\>d\<V,E\>),\,\,\E\>\nonumber\\
&&\quad+\<V,QE\>\<d\<V,E\>\wedge d(\theta^2\<\nabla\n QE,QE\>),\,\,\E\>+\frac12\<d(\theta^2\<\nabla\n QE,QE\>|V|^2[QE,E]),\,\,\E\>\nonumber\\
&&\quad-\frac12|V|^2\<d(\theta^2\<\nabla\n QE,QE\>[QE,E]),\,\,\E\>\nonumber\\
&&\geq \theta^2Q^*\Pi(\ii_EDV,\ii_EDV)+\<d (\theta^2\<\nabla\n QE,QE\>\<V,QE\>d\<V,E\>),\,\,\E\>\nonumber\\
&&\quad+\frac12\<d(\theta^2\<\nabla\n QE,QE\>|V|^2[QE,E]),\,\,\E\>-C(|V||\theta DV|+|V|^2).\nonumber\eeq
We integrate the above inequality over $\Om$ to have
\beq&&\int_\Om \theta^2\<(\div_gQPV)QE-(\div_gV)QPQE,\ii_EDV\>dg\nonumber\\
&&\geq\int_\Om\theta^2Q^*\Pi(\ii_EDV,\ii_EDV)dg+\Big(d (\theta^2\<\nabla\n QE,QE\>\<V,QE\>d\<V,E\>),\,\,\E\Big)_{\LL^2(\Om,\Lambda^2)}\nonumber\\
&&\quad+\frac12\Big(d(\theta^2\<\nabla\n QE,QE\>|V|^2[QE,E]),\,\,\E\Big)_{\LL^2(\Om,\Lambda^2)}-C\int_\Om(|V||\theta DV|+|V|^2)dg\nonumber\\
&&=\int_\Om\theta^2Q^*\Pi(\ii_EDV,\ii_EDV)dg+\Big(\theta^2\<\nabla\n QE,QE\>\<V,QE\>d\<V,E\>,\,\,\delta\E\Big)_{\LL^2(\Om,\Lambda^2)}\nonumber\\
&&\quad+\frac12\Big(\theta^2\<\nabla\n QE,QE\>|V|^2[QE,E],\,\,\delta\E\Big)_{\LL^2(\Om,\Lambda^2)}-C\int_\Om(|V||\theta DV|+|V|^2)dg\nonumber\\
&&\quad+\int_{\pl\Om}\theta^2\<\nabla\n QE,QE\>\<\nu\wedge(\<V,QE\>d\<V,E\>+\frac12|V|^2[QE,E],\,\,\E\>d\Ga,\nonumber\eeq where the following formula has been used
$$(d Y,T)_{\LL^2(\Om,\Lambda^2)}=(Y,\delta T)_{\LL^2(\Om,T)}+\int_\Ga\<\nu\wedge Y,T\>d\Ga\qfq Y\in TM,\quad T\in\Lambda^2M,$$ see \cite[Theorem 1.29]{Yaobook}.
Thus (\ref{4.3}) follows.
\end{proof}

{\bf Proof of Theorem \ref{t3.1}.}\,\,\,We shall prove the case of $m=1.$ The cases of $m\geq2$ can be treated by repeating a similar argument.

Let $F=e^{-\g\kappa}(f_1QX+f_2\nabla\n X)$ be given with $f_1,$ $f_2\in\WW^{1,2}(\Om_{-\varepsilon})$ and let $V\in\LL^2(\Om_{-\varepsilon},T)$ be the solution to problem (\ref{3.1}). By Proposition \ref{p3.1},
$$D_\tau V\in\LL^2(\Om_{-\varepsilon},T),\quad \tau=\frac{\a_t}{|\a_t|},$$ since $\Phi=-e^\var\frac{\<D\psi,Q\a_t\>}{|\a_t|^2}\a_t.$ Noting that $\div_gQ\nabla\n=0$ and $Q\tau,$ $\tau$ forms an orthonormal frame on $\overline{\Om}_{-\varepsilon}$ with positive orientation, from (\ref{f3.3}) and (\ref{2.5}), we have
\be\left\{\begin{array}{l}\<D_{Q\tau}V,Q\tau\>=\eta\<V,\nabla\n X\>+f_2-\<D_\tau V,\tau\>\qfq x\in\Om_{-\varepsilon},\\
\<D_{Q\tau}V,\tau\>\Pi(\tau,\tau)=(\<D_{Q\tau}V,Q\tau\>+\<D\tau V,\tau\>)\Pi(\tau,Q\tau)+\<D_\tau V,Q\tau\>\Pi(Q\tau,Q\tau)\\
\quad+f_1+\eta\<V,QX\>\qfq x\in\Om_{-\varepsilon},\end{array}\right.\ee which yield, by (\ref{Pi2.80}),
$$\|D_{Q\tau}V\|^2_{\LL^2(\Om_{-\varepsilon},T)}\leq C(\|D_\tau V\|^2_{\LL^2(\Om_{-\varepsilon},T)}+\|V\|^2_{\LL^2(\Om_{-\varepsilon},T)}+\|F\|^2_{\LL^2(\Om_{-\varepsilon},T)}).$$
Thus $V\in\WW^{1,2}(\Om_{-\varepsilon},T).$  In the case of $m=1$ (\ref{V3.2}) follows    from (\ref{F3.26}) and (\ref{V2.45}). \hfill$\Box$\\

{\bf Proof of Theorem \ref{tt4.1}.}\,\,\,We shall prove the case of $m=1.$ The cases of $m\geq2$ can be treated by repeating a similar argument. By Theorem \ref{t3.1}, we may assume that $F\in\CC^1(\Om_{-\varepsilon},T).$ We take two cut-off functions $\theta_1,$ $\theta_2\in\CC^1(-\varepsilon,b_1)$ as follows:
$$ \theta_1(s)=1\qfq s\in(\varepsilon/3,b_1),\quad \theta_1(s)=0\qfq s\in(-\varepsilon,\varepsilon/4);$$
$$ \theta_2(s)=1\qfq s\in(-\varepsilon,\epsilon/3),\quad \theta_2(s)=0\qfq s\in(\varepsilon/2,b_1).$$

Let $\si_0>0$ be given such that
$$Q^*\Pi\geq3\si_0\id\qfq x\in S^+\setminus S_{\varepsilon/3}.$$
We take $\si=\si_0,$ $\theta=\theta_1,$ $E=\a_t/|\a_t|,$ and $E=Q\a_t/|\a_t|$ in (\ref{4.3}), respectively. Then we obtain
\beq&&\int_{\Om_{-\varepsilon}}\theta^2[Q^*\Pi(\ii_EDV,\ii_EDV)+Q^*\Pi(\ii_{QE}DV,\ii_{QE}DV)-2\si_0|DV|^2]dg\nonumber\\
&&\leq C(\|F\|^2_{\LL^2(\Om_{-\varepsilon},T)}+\|V\|^2_{\LL^2(\Om,T)}+\|q\|^2_{\WW^{1,2}(\Ga_{b_1})}).\nonumber\eeq Thus
\be \|V\|^2_{\WW^{1,2}(\Sigma^{b_1}_{\varepsilon/3},T)}\leq C(\|F\|^2_{\LL^2(\Om_{-\varepsilon})}+\|V\|^2_{\LL^2(\Om,T)}+\|q\|^2_{\WW^{1,2}(\Ga_{b_1})}).\label{4.10}\ee

Set
$$V_2=\theta_2V.$$ Then  $V_2$ satisfies problem
\be\left\{\begin{array}{l}\L_0V_2=\theta_2F+e^{-\g\kappa}(\<D\theta_2,Q\nabla\n V\>QX+\<D\theta_2,\nabla\n X\>\nabla\n X)\qfq x\in \Sigma_{-\varepsilon}^{\varepsilon/2},\\
\<V_2,QX\>|_{\Ga_{-\varepsilon}}=p,\quad\<V_2,\nu\>|_{\Ga_{\varepsilon/2}}=0.\end{array}\right.\label{4.11}\ee

Noting that $\supp D\theta_2\subset(\varepsilon/3,\varepsilon/2),$ applying Theorem \ref{t3.1} to problem (\ref{4.11}) yields, by (\ref{4.10}),
\beq\|\theta_2 V\|^2_{\WW^{1,2}(\Sigma_{-\varepsilon}^{\varepsilon/2},T)}&&\leq C(\|F\|^2_{\WW^{1,2}( \Sigma_{-\varepsilon}^{\varepsilon/2},T)}+\|V\|^2_{\WW^{1,2}( \Sigma_{-\varepsilon}^{\varepsilon/2},T)}+ \|p\|^2_{\WW^{1,2}(\Ga_{-\varepsilon})})\nonumber\\
&&\leq C(\|F\|^2_{\LL^2(\Om_{-\varepsilon},T)}+\|F\|^2_{\WW^{1,2}(\Sigma_{-\varepsilon}^{\varepsilon/2},T)}
+\|p\|^2_{\WW^{1,2}(\Ga_{-\varepsilon})}).\label{4.15}\eeq
Thus (\ref{4.101}) follows from (\ref{4.10}) and (\ref{4.15}). \hfill$\Box$

\setcounter{equation}{0}
\def\theequation{5.\arabic{equation}}
\section{$\WW^{m,2}$ Solutions of a coupling tensor system of mixed type }
\hskip\parindent  For given $(F,f)\in\LL^2(\Om_{-\varepsilon},T)$ arbitrarily, we directly seek some appropriate conditions on the boundary data
$(p,q)\in\LL^2(\Ga_{-\varepsilon})\times\LL^2(\Ga_{b_1})$ such that problem
\be\left\{\begin{array}{l}Dv=\nabla\n V+F\qfq x\in\Om_{-\varepsilon},\\
\div_gV=\rho v+f\qfq x\in\Om_{-\varepsilon},\\
\<V,QX\>|_{\Ga_{-\varepsilon}}=p,\quad\<V,\nu\>|_{\Ga_{b_1}}=q,\quad
\int_{\Om_{-\varepsilon}}vdg=0,\end{array}\right.\nonumber\ee where $\rho=-\tr_g\Pi,$
admits a solution $(V,v)$ on $\Om_{-\varepsilon}.$

 For convenience, we denote
\be\cc V=\eta e^{-\g\kappa}(\<V,QX\>QX+\<V,\nabla\n X\>\nabla\n X)\qfq V\in\LL^2(\Om_{-\varepsilon},T),\label{n5.2}\ee where the function $\eta$ and the vector field $X$ are given in (\ref{2.32*}) and (\ref{2.32}), respectively.

We consider a uniqueness problem.

\begin{pro}\label{p5.1} Let $S$ be of $\CC^5.$ Problem
\be\begin{cases}\L_0^*W+\cc W=0\qfq x\in\Om_{-\varepsilon},\\
\quad W|_{\Ga_{b_1}}=0
\end{cases}\ee
admits a unique zero solution, where $\L_0^*$ is given in $(\ref{2.13}).$
\end{pro}

\begin{proof}
Let
$$w_1=e^{-\g\kappa}\<W,QX\>,\quad w_2=e^{-\g\kappa}\<W,\nabla\n X\>.$$
From (\ref{2.13}), we have
$$\L_0^*W=\nabla\n QDw_1-Dw_2-\cc W,$$ that is,
\be \nabla\n QDw_1-Dw_2=0\qfq x\in\Om_{-\varepsilon}.\label{5.3}\ee
Thus we obtain
$$\div_gQ\nabla\n QDw_1=0,\quad \div_g(\nabla\n)^{-1}Dw_2=0\qfq x\in S_+.$$
In addition $W|_{\Ga_{b_1}}=0$ implies
$$w_1|_{\Ga_{b_1}}=w_2|_{\Ga_{b_1}}=0.$$ Then
$$Dw_i=\<Dw_i,\nu\>\nu\qfq x\in\Ga_{b_1}.$$ It follows from (\ref{5.3}) that
$$\<Dw_1,\nu\>\nabla\n Q\nu=\<Dw_2,\nu\>\nu\qfq x\in\Ga_{b_1},$$ this is,
$$\<Dw_1,\nu\>=\<Dw_2,\nu\>=0\qfq x\in\Ga_{b_1},$$ since $\Pi(Q\nu,Q\nu)>0$ for $x\in\Ga_{b_1}.$ Noting that $-\div_gQ\nabla\n QDw_1$ and $\div_g(\nabla\n)^{-1}Dw_2$ are elliptic in $S_+,$ the uniqueness theorem in \cite{Ped} implies that
\be w_1=w_2=0\qfq x\in S_+.\label{5.4}\ee

Next, we need to prove that when given $\varepsilon>0$ is small,
\be w_1=w_2=0\qfq x\in \Sigma_{-\varepsilon}^0,\label{5.5}\ee where
$$\Sigma_{-\varepsilon}^0=\{\,\a(t,s)\,|\,(t,s)\in[0,a)\times(-\varepsilon,0)\,\}.$$
Let $X_1$ and $X_2$ be given in (\ref{X2.82}). Set
$$H=\frac1{\kappa^2}X_2\qfq x\in\Sigma_{-\varepsilon}^0.$$ From (\ref{5.3}), we have
\beq\<QDw_2,D(Hw_1)\>&&=\<Q\nabla\n QDw_1,D(Hw_1)\>=\<D_{Q\nabla\n QDw_1}H,Dw_1\>+\<D_HDw_1,Q\nabla\n QDw_1\>\nonumber\\
&&=\<D_{Q\nabla\n QDw_1}H,Dw_1\>+\div_g(\<Dw_1,Q\nabla\n QDw_1\>H)\nonumber\\
&&\quad-\<Dw_1,D_H(Q\nabla\n QDw_1)\>-\<Dw_1,Q\nabla\n QDw_1\>\div_gH\nonumber\\
&&=2\<D_{Q\nabla\n QDw_1}H,Dw_1\>+\<QDw_1,D_H(\nabla\n) QDw_1)\>\nonumber\\
&&\quad-(\<Q\nabla\n QDw_1,D_HDw_1)\>+\<D_{Q\nabla\n QDw_1}H,Dw_1\>)\nonumber\\
&&\quad+\<QDw_1,\nabla\n QDw_1\>\div_gH-\div_g(\<QDw_1,\nabla\n QDw_1\>H),\nonumber\eeq that is,
\be 2\<QDw_2,D(Hw_1)\>=\Upsilon(w_1)-\div_g(\<QDw_1,\nabla\n QDw_1\>H)\qfq x\in\Sigma_{-\varepsilon}^0,\label{5.6}\ee where
$$\Upsilon(w_1)=2\<D_{Q\nabla\n QDw_1}H,Dw_1\>+\<QDw_1,D_H(\nabla\n) QDw_1)\>+\<QDw_1,\nabla\n QDw_1\>\div_gH.$$
Moreover, using $H=\kappa^{-2}X_2,$ we compute
\beq\kappa^3\<D_{Q\nabla\n QDw_1}H,Dw_1\>&&=-2\<Q\nabla\n QDw_1,D\kappa\>X_2(w_1)+\kappa\<D_{Q\nabla\n QDw_1}X_2,Dw_1\>\nonumber\\
&&=2[\lam_2X_1(w_1)X_1(\kappa)+\lam_1X_2(w_1)X_2(\kappa)]X_2(w_1)\nonumber\\
&&\quad-\lam_2\kappa\<D_{X_1}X_2,X_1\>[X_1(w_1)]^2-\lam_1\kappa\<D_{X_2}X_2,X_1\>X_1(w_1)X_2(w_2)\nonumber\\
&&=\lam_2\O(\kappa)[X_1(w_1)]^2+2\lam_1X_2(\kappa)[X_2(w_1)]^2+\O(\kappa)X_1(w_1)X_2(w_2),\nonumber\eeq where the formula $\<D_{X_2}X_2,X_2\>=0$ is used,
\beq\kappa^3\<QDw_1,D_H(\nabla\n) QDw_1)\>&&=\kappa\<QDw_1,D_{X_2}(\nabla\n) QDw_1)\>=\kappa X_2(\lam_2)[X_1(w_1)]^2\nonumber\\
&&\quad+\O(\kappa)\{[X_2(w_1)]^2+X_1(w_1)X_2(w_1)\},\nonumber\eeq where the formula $D_{X_2}(\nabla\n)(X_2,X_2)=X_2(\lam_2)$ is used, and
\beq\kappa^3\<QDw_1,\nabla\n QDw_1\>&&\div_gH=\kappa\{\lam_2[X_1(w_1)]^2+\lam_1[X_2(w_1)]^2\}\div_gX_2\nonumber\\
&&\quad-2\{\lam_2[X_1(w_1)]^2+\lam_1[X_2(w_1)]^2\}X_2(\kappa)\nonumber\\
&&=-\lam_2[2X_2(\kappa)+\O(\kappa)][X_1(w_1)]^2-[2\lam_1X_2(\kappa)+\O(\kappa)][X_2(w_1)]^2,\nonumber\eeq respectively.
Since $\kappa X_2(\lam_2)=\lam_2[X_2(\kappa)+\O(\kappa)],$ we obtain
\beq\kappa^3\Upsilon(w_1)&&=-\lam_2[X_2(\kappa)+\O(\kappa)][X_1(w_1)]^2+[2\lam_1X_2(\kappa)+\O(\kappa)][X_2(w_1)]^2\nonumber\\
&&+\O(\kappa)X_1(w_1)X_2(w_2),\nonumber\eeq which yields
\be   -\Upsilon(w_1)\geq\frac{c}{\kappa^2}|Dw_1|^2\qfq x\in\Sigma_{-\varepsilon}^0, \label{5.7}\ee when $c>0$ and $\varepsilon>0$ are small enough, since $X_2(\kappa)>0$ by (\ref{as}).

Since $S$ is of $\CC^5,$ by Theorem \ref{tt4.2} $W\in\WW^{5,2}(\Om_{\varepsilon})$ and thus
$$w_1,\quad w_2\in\CC^3(\Om_{-\varepsilon},T).$$ From (\ref{5.4}) we have
$$\lim_{s\rw0^-}\frac{Dw_1}{\kappa^2}(\a(t,s))=\lim_{s\rw0^-}\frac{Dw_2}{\kappa^2}=0\qfq t\in[0,a).$$ Thus
\be \lim_{s\rw0^-}[2H(w_1)QDw_2+\<QDw_1,\nabla\n QDw_1\>H](\a(t,s))=0\qfq t\in[0,a).\label{5.8}\ee
Moreover, since $\div_gQDw_2=0,$ we have
\be\<QDw_2,D(Hw_1)\>=\div_g(Hw_1QDw_2).\label{5.9}\ee For given $x\in\Ga_{-\varepsilon}$ and by (\ref{5.3}), we obtain
\beq&&\kappa^2[2H(w_1)\<QDw_2,\nu\>+\<QDw_1,\nabla\n QDw_1\>\<H,\nu\>]\nonumber\\
&&=\<X_2,\nu\>\{\lam_2[X_1(w_1)]^2-\lam_1[X_2(w_1)]^2\}-2\lam_2\<X_1,\nu\>X_1(w_1)X_2(w_1),\nonumber\eeq that implies, by (\ref{nu2.30}),
\be \kappa^2[2H(w_1)\<QDw_2,\nu\>+\<QDw_1,\nabla\n QDw_1\>\<H,\nu\>]\geq c_1|Dw_1|^2\qfq x\in\Ga_{-\varepsilon},\label{5.10}\ee when given $\varepsilon>0$ is small enough.

For given $-\varepsilon<s<0,$ we integrate $\Upsilon(w_1)$ over $\Sigma_{-\varepsilon}^s$ to have, by (\ref{5.7}), (\ref{5.6}) and (\ref{5.9}),
$$ c\int_{\Sigma_{-\varepsilon}^s}\frac{|Dw_1|^2}{\kappa^2}dg\leq-\int_{\Ga_{-\varepsilon}\cup\Ga_s}[2H(w_1)\<QDw_2,\nu\>+\<QDw_1,\nabla\n QDw_1\>\<H,\nu\>]d\Ga,$$
from which we obtain, (\ref{5.8}) and (\ref{5.10}),
$$c\int_{\Sigma_{-\varepsilon}^0}\frac{|Dw_1|^2}{\kappa^2}dg\leq-\int_{\Ga_{-\varepsilon}}[2H(w_1)\<QDw_2,\nu\>+\<QDw_1,\nabla\n QDw_1\>\<H,\nu\>]d\Ga\leq0,$$ that is, (\ref{5.5}) is true.
\end{proof}

Let
\be\W_0=\{\,W\in\LL^2(\Om_{-\varepsilon},T)\,|\,\L_0^*W+\cc W=0,\,\<W,QX\>|_{\Ga_{b_1}}=\<W,\nu\>|_{\Ga_{-\varepsilon}}=0\,\}.\label{n5.12}\ee Set
$$\Ga(\W_0)=\{\,\pounds_1\<W,\nu\>|_{\Ga_{b_1}}\,|\,W\in\W\,\},$$ where $\pounds_1$ is given in (\ref{n2.17x}).
Consider the direct sum composition
$$\LL^2(\Ga_{b_1})=\Ga(\W_0)\oplus\Ga^\bot(\W_0)\quad\mbox{in the norm of $\LL^2(\Ga_{b_1}).$}$$

\begin{lem}\label{l5.1} Problem
\be\begin{cases}\L_0V+\cc V=0,\\
\<V,\nu\>|_{\Ga_{b_1}}=q,\quad\<V,QX\>|_{\Ga_{-\varepsilon}}=0
\end{cases}\label{5.11}\ee admits a solution $V\in\LL^2(\Om_{-\varepsilon},T)$ if and only if
$$q\in\Ga^\bot(\W_0).$$
\end{lem}

\begin{proof} Let $V$  satisfy  (\ref{5.11}) and $W\in \W_0.$ From (\ref{2.14}), we have
\beq &&-(W,\cc V)_{\LL^2(\Om,T)}=-(V,\cc W)_{\LL^2(\Om,T)}\nonumber\\
&&\quad+\int_{\Ga_{b_1}\cup\Ga_{-\varepsilon}}(\pounds_1\<V,\nu\>\<W,\nu\>-\pounds_2\<V,QX\>\<W,QX\>)d\Ga,\nonumber\eeq which yield
$q\in\Ga^\bot(\W_0).$

Conversely, for given $q\in\Ga^\bot(\W_0),$ by Theorem \ref{t2.1} there exists a unique solution $V_0\in\LL^2(\Om_{-\varepsilon},T)$ to problem
$$\begin{cases}\L_0V_0=0,\\
\<V_0,\nu\>|_{\Ga_{b_1}}=q,\quad\<V_0,QX\>|_{\Ga_{-\varepsilon}}=0.
\end{cases}$$

For given $Z\in\LL^2(\Om_{-\varepsilon},T),$ by Theorem \ref{t2.1} problem
$$\begin{cases}\L_0V=-\cc Z,\\
\<V,\nu\>|_{\Ga_{b_1}}=\<V,QX\>|_{\Ga_{-\varepsilon}}=0
\end{cases}$$ admits a unique solution $V\in\LL^2(\Om_{-\varepsilon},T).$ By Theorem \ref{tt4.1}
$$V\in\WW^{1,2}(\Om_{-\varepsilon},T).$$ We define a linear operator $\B:$ $\LL^2(\Om_{-\varepsilon},T)\rw\LL^2(\Om_{-\varepsilon},T)$ by
$$\B Z=V.$$ Then $\B$ is compact on $\LL^2(\Om_{-\varepsilon},T).$ By Fredholm's theorem, problem
\be(\id-\B)V=V_0\label{5.13}\ee admits a solution in $\LL^2(\Om_{-\varepsilon},T)$ if and only if
\be(V_0,U)_{\LL^2(\Om_{-\varepsilon},T)}=0\quad\mbox{for all}\quad U\in\Z(\id-\B^*).\label{5.12}\ee In addition, it is easy to check that
$$\Z(\id-\B^*)=\{\,U\in\LL^2(\Om_{-\varepsilon},T)\,|\,U+\cc(\L_0^*)^{-1}U=0\,\}.$$
For $U\in\Z(\id-\B^*),$ let $W=(\L_0^*)^{-1}U.$ Then
$$\L_0^*W+\cc W=0,$$ that is, $W\in\Ga(\W_0).$ Thus (\ref{5.12}) is equivalent to
\beq0&&=(V_0,\L_0^*W)_{\LL^2(\Om_{-\varepsilon},T)}=(\L_0V_0,W)_{\LL^2(\Om_{-\varepsilon},T)}-\int_{\Ga_{b_1}}\pounds_1q\<W,\nu\>d\Ga\nonumber\\
&&=-\int_{\Ga_{b_1}}\pounds_1q\<W,\nu\>d\Ga,\nonumber\eeq that is, $q\in\Ga^\bot(\W_0).$

Then there exists a solution $V\in\LL^2(\Om_{-\varepsilon},T)$ to problem (\ref{5.13}) that is also a solution to problem (\ref{5.11}).
\end{proof}

Set
$$\V_0=\{\,V\in\LL^2(\Om_{-\varepsilon},T)\,|\,\<V,\nu\>|_{\Ga_{b_1}}=\<V,QX\>|_{\Ga_{-\varepsilon}}=0\,\}.$$
By a similar argument and the duality, we have the following.
\begin{lem}\label{l5.2} Problem
\be\begin{cases}\L_0^*W+\cc W=0,\\
\<W,\nu\>|_{\Ga_{-\varepsilon}}=0,\quad\<W,QX\>|_{\Ga_{b_1}}=p
\end{cases}\label{5.14}\ee admits a solution $W\in\LL^2(\Om_{-\varepsilon},T)$ if and only if
$$\int_{\Ga_{-\varepsilon}}\pounds_1p\<V,\nu\>d\Ga=0\qfq V\in\V_0.$$
\end{lem}

\begin{pro}\label{p5.2} Let $S$ be of $\CC^{m+5}.$ There exists a $H_0\in\WW^{m+5,2}(\Om_{-\varepsilon},T)$ satisfying
\be\begin{cases}\div_gQ\nabla\n H_0=\div_gH_0=0\qfq x\in\Om_{-\varepsilon},\\
\<H_0,QX\>|_{\Ga_{-\varepsilon}}=0,\quad \int_{\Ga_{-\varepsilon}}\<\nabla\n H_0,Q\nu\>d\Ga=1.\end{cases}\ee
\end{pro}
\begin{proof} {\bf Step 1.}\,\,\,We consider a similar problem on a bigger region as
$$\hat\Om_{-\varepsilon}=\{\,\a(t,s)\,|\,(t,s)\in[0,a)\times(-\varepsilon,b_2)\,\},$$ where $b_2>b_1$ such that
$$\kappa(\a(t,s))>0\qfq (t,s)\in[0,a)\times[b_1,b_2].$$
We claim that there is a solution   $H_0\in\LL^2(\Om_{-\varepsilon},T)$ to problem
\be\begin{cases}\L_0H_0+\cc H_0=0\qfq x\in\hat\Om_{-\varepsilon},\\
\<H_0,\nu\>|_{\Ga_{b_2}}=q,\quad\<H_0,QX\>|_{\Ga_{-\varepsilon}}=0\label{n5.16}
\end{cases}\ee such that
$$\int_{\Ga_{-\varepsilon}}\<\nabla\n H_0,Q\nu\>d\Ga=1.$$
By contradiction. We suppose that for all $q\in\Ga^\bot(\W_0)$ solutions $V$ to problem (\ref{n5.16}) satisfy
\be\int_{\Ga_{-\varepsilon}}\<\nabla\n V,Q\nu\>d\Ga=0.\label{5.16}\ee

Set
$$p_0=\frac{e^{\g\kappa}\Pi(X,Q\nu)}{\Pi(X,X)}\qfq x\in\Ga_{-\varepsilon}.$$ From (\ref{2.35x}), $p_0\not=0$ for $x\in\Ga_{-\varepsilon}.$ By (\ref{2.4})
$$\<X,\nu\>V=\<V,QX\>Q\nu+\<V,\nu\>X\qfq x\in\Ga_{-\varepsilon},\quad V\in\V_0.$$
Then $V\in\V_0$ and (\ref{5.16})  imply that
$$\int_{\Ga_{-\varepsilon}}\pounds_1p_0\<V,\nu\>d\Ga=\int_{\Ga_{-\varepsilon}}\frac{\Pi(X,Q\nu)}{\<X,\nu\>}\<V,\nu\>d\Ga
=\int_{\Ga_{-\varepsilon}}\<\nabla\n V,Q\nu\>d\Ga=0$$ for all $V\in\V_0,$ where $\pounds_1$ is given in (\ref{n2.17x}). By Lemma \ref{l5.2} problem (\ref{5.14}) admits a solution $W_0\in\LL^2(\Om_{-\varepsilon},T)$ with
$p=p_0,$ where $\Om_{-\varepsilon}$ and $b_1$ are replaced with $\hat\Om_{-\varepsilon}$ and $b_2,$ respectively.

Now we claim that
\be\int_{\Ga_+}\pounds_1\<W_0,\nu\>qd\Ga=0\quad\mbox{for all}\quad q\in\Ga^\bot(\W_0).\label{5.17}\ee By Lemma \ref{l5.1}, for each $q\in\Ga^\bot(\W_0),$ problem
(\ref{n5.16}) admits a solution $V.$ Then (\ref{5.17}) follows  from the formula (\ref{2.14}).

The formula (\ref{5.17}) implies that
$$\pounds_1\<W_0,\nu\>\Big|_{\Ga_{b_2}}\in\Ga(\W_0),$$ that is, there exists $W_1\in\W_0$ such that
$$\<W_1,\nu\>=\<W_0,\nu\>\quad x\in\Ga_{b_2}.$$ Since $\<W_1,QX\>=\<W_0,QX\>=0$ for $x\in\Ga_{b_2},$ we have
$$W_0=W_1\qfq x\in\Ga_{b_2}.$$ By Proposition \ref{p5.1}, $W_1=W_0$ for all $x\in\hat\Om_{-\varepsilon}.$ This is a contradiction since $\<W_1,\nu\>|_{\Ga_{-\varepsilon}}
=\<W_0,\nu\>|_{\Ga_{-\varepsilon}}=p_0\not=0.$ The proof is complete.

{\bf Step 2.}\,\,\,Next, we show that $H_0\in\WW^{1,2}(\Om_{-\varepsilon},T).$ Let $\var\in\CC^{m+5}(\hat\Om_{-\varepsilon})$ be a cutoff function satisfying
$$\var(x)=1\qfq x\in\Om_{-\varepsilon};\quad\var(x)=0\qfq x\in\Ga_{b_2}.$$ Set
$$H_1=\var H_0.$$ Then $H_1\in\LL^2(\hat\Om_{-\varepsilon},T)$ solves problem
$$\begin{cases}\L_0H_1+\cc H_1=\<Q\nabla\n H_0,D\var\>QX+\<H_0,D\var\>\nabla\n X\qfq x\in\hat\Om_{-\varepsilon},\\
\<H_1,\nu\>|_{\Ga_{b_2}}=\<H_1,QX\>|_{\Ga_{-\varepsilon}}=0.\end{cases}$$ Ii follows from Theorem \ref{tt4.1} that $H_1\in\WW^{m+5,2}(\hat\Om_{-\varepsilon},T).$ The proof is complete.
\end{proof}

Let
$$\dot{\WW}^{1,2}(\Om_{-\varepsilon})=\{\,w\in\WW^{1,2}(\Om_{-\varepsilon})\,|\,\int_{\Om_{-\varepsilon}}wdg=0\,\}$$ with the norm
$$\|w\|_{\dot{\WW}^{1,2}(\Om_{-\varepsilon})}=\|Dw\|_{\LL^2(\Om_{-\varepsilon},T)}.$$

For given $V\in\LL^2(\Om_{-\varepsilon},T),$ we solve problem
\be\begin{cases}\Delta_g\varrho=\eta\<V,QX\>\qfq x\in\Om_{-\varepsilon},\\
\varrho|_{\Ga_{-\varepsilon}}=\varrho|_{\Ga_{b_1}}=0,\end{cases}\label{5.27}\ee and define an operator $\varrho_{\cdot}:$ $\LL^2(\Om_{-\varepsilon},T)\rw\WW^{2,2}(\Om_{-\varepsilon})$ by
$$\varrho_V=\varrho.$$

For given $(V,v)\in\LL^2(\Om_{-\varepsilon},T)\times\dot{\WW}^{1,2}(\Om_{-\varepsilon}),$ we solve problem
\be\begin{cases}\L_0W+\cc V=e^{-\g\kappa}\rho v\nabla\n X,\\
\<W,\nu\>|_{\Ga_{b_1}}=\<W,QX\>|_{\Ga_{-\varepsilon}}=0,\end{cases}\label{5.20}\ee by Theorem \ref{t2.1}, to have $W\in\LL^2(\Om_{-\varepsilon},T),$ where $\rho=-\tr_g\Pi.$ Moreover, by Theorem \ref{tt4.1}
$$W\in\WW^{1,2}(\Om_{-\varepsilon},T).$$ Then we have
$$\div_gQ\nabla\n W=\eta\<W-V,QX\>=\Delta_g\varrho_{W-V}.$$ Thus we obtain
\be\div_gQ[\nabla\n W+QD\varrho_{W-V}+\phi\nabla\n H_0]=0\qfq x\in\Om_{-\varepsilon},\label{5.28}\ee where $H_0\in\WW^{1,2}(\Om_{-\varepsilon},T)$ is given in Proposition \ref{p5.2} and
$$\phi=-\int_{\Ga_{-\varepsilon}}\<\nabla\n W+QD\varrho_{W-V},Q\nu\>d\Ga.$$
Since (\ref{5.28}) means that $d[\nabla\n W+QD\varrho_{W-V}+\phi\nabla\n H_0]=0$ where $d$ is the exterior derivative, by the de Rham cohomology group theorem
 there exists a unique $w\in\dot{\WW}^{1,2}(\Om_{-\varepsilon})$ such that
\be Dw=\nabla\n W+QD\varrho_{W-V}+\phi\nabla\n H_0.\label{5.19}\ee  Clearly, $w\in\WW^{2,2}(\Om_{-\varepsilon}).$ We define a linear operator
$\A:$ $\LL^2(\Om_{-\varepsilon},T)\times\dot{\WW}^{1,2}(\Om_{-\varepsilon})\rw\LL^2(\Om_{-\varepsilon},T)\times\dot{\WW}^{1,2}(\Om_{-\varepsilon})$ by
\be\A(V,v)=(W,w).\label{n5.23}\ee Then $\A$ is  compact on $\LL^2(\Om_{-\varepsilon},T)\times\dot{\WW}^{1,2}(\Om_{-\varepsilon}).$

We consider the structure of
$$\Z(\id-\A^*)=\{\,(U,u)\in\LL^2(\Om_{-\varepsilon},T)\times\dot{\WW}^{1,2}(\Om_{-\varepsilon})\,|\,(U,u)=\A^*(U,u)\,\}.$$ Clearly, $\dim\Z(\id-\A^*)<\infty.$ Suppose that $W_0$  satisfies problem
\be
\begin{cases}
\L^*_0W_0=\eta\varrho_0 QX,\\
\<W_0,\nu\>\big|_{\Ga_{-\varepsilon}}=\frac{e^{\g\kappa}\Pi(X,Q\nu)}{\Pi(X,X)},\quad
\<W_0,QX\>\big|_{\Ga_{b_1}}=0,
\end{cases}\label{5.23}\ee where
\be\begin{cases}\Delta_g\varrho_0=0,\\
\varrho_0|_{\Ga_{b_1}}=0,\quad\varrho_0|_{\Ga_{-\varepsilon}}=1.
\end{cases}\label{5.25}\ee

\begin{lem} Let $H_0$ be given in Proposition $\ref{p5.2}.$ Then $(U,u)\in\Z(\id-\A^*)$ if and only if $(U,u)$ satisfies
\be\begin{cases}U=-\cc(\L_0^*)^{-1}(U+\nabla\n Du)+c(u)(\cc W_0+\eta\varrho_0QX),\\
-\Delta_gu=e^{-\g\kappa}\rho \<(\L_0^*)^{-1}(U+\nabla\n Du),\nabla\n X\>-e^{-\g\kappa}c(u)\rho\<W_0,\nabla\n X\>,\\
\<Du,\nu\>|_{\pl\Om_{-\varepsilon}}=0,\quad\int_{\Om_{-\varepsilon}}udg=0,
\end{cases}\label{5.27*}\ee
where $c(u)=(\nabla\n H_0,Du)_{\LL^2(\Om_{-\varepsilon},T)}.$
\end{lem}

\begin{proof}
Let $W$ solve problem (\ref{5.20}). It follows from the formula $\<X,\nu\>W=\<W,QX\>Q\nu+\<W,\nu\>X$ that
$$\<\nabla\n W,Q\nu\>=\frac{\Pi(X,Q\nu)}{\<X,\nu\>}\<W,\nu\>\qfq x\in\Ga_{-\varepsilon}.$$ By (\ref{2.14}) we obtain
\beq\int_{\Ga_{-\varepsilon}}\<\nabla\n W,Q\nu\>d\Ga&&=\int_{\Ga_{-\varepsilon}}\pounds_1\<W,\nu\>\<W_0,\nu\>d\Ga\nonumber\\
&&=(W_0,\L_0W)_{\LL^2(\Om_{-\varepsilon},T)}-(W,\L_0^*W_0)_{\LL^2(\Om_{-\varepsilon},T)}.\label{5.21}\eeq
In addition, from (\ref{5.25}) we have
\be\int_{\Ga_{-\varepsilon}}\<D\varrho_{W-V},\nu\>d\Ga=\int_{\Om_{-\varepsilon}}\div_g\varrho_0D\varrho_{W-V}dg=\int_{\Om_{-\varepsilon}}\varrho_0\eta\<W-V,QX\>dg.\label{5.22}\ee
It follows from (\ref{5.20})-(\ref{5.22}) that
\beq\phi&&=\int_{\Om_{-\varepsilon}}[\eta\varrho_{0}\<V,QX\>+\<V, \cc W_0\>-e^{-\g\kappa}\rho v\<W_0,\nabla\n X\>]dg\nonumber\\
&&=(V,\cc W_0+\eta\varrho_0QX)_{\LL^2(\Om_{-\varepsilon},T)}+(v,-e^{-\g\kappa}\rho\<W_0,\nabla\n X\>)_{\LL^2(\Om_{-\varepsilon})}.\label{5.26}\eeq

 For given $(V,v),$ noting that
$$\int_{\Om_{-\varepsilon}}\<QD\varrho_{W-V},Du\>dg=-\int_{\Om_{-\varepsilon}}\div_g(\varrho_{W-V}QDu)dg=0,$$ using (\ref{5.19}), (\ref{5.20}), and (\ref{5.26}), we have
\beq&&(V,U)_{\LL^2(\Om_{-\varepsilon},T)}+(v,-\Delta_gu)_{\LL^2(\Om_{-\varepsilon})}+\int_{\pl\Om_{-\varepsilon}}v\<Du,\nu\>d\Ga
=\Big((V,v),(U,u)\Big)_{\LL^2(\Om_{-\varepsilon},T)\times\dot{\WW}^{1,2}(\Om_{-\varepsilon})}\nonumber\\
&&=\Big(\A(V,v),(U,u)\Big)_{\LL^2(\Om_{-\varepsilon},T)\times\dot{\WW}^{1,2}(\Om_{-\varepsilon})}
=(W,U)_{\LL^2(\Om_{-\varepsilon},T)}+(Dw,Du)_{\LL^2(\Om_{-\varepsilon},T)}\nonumber\\
&&=(W,U+\nabla\n Du)_{\LL^2(\Om_{-\varepsilon},T)}+(\phi\nabla\n H_0,Du)_{\LL^2(\Om_{-\varepsilon},T)}\nonumber\\
&&=\Big((\L_0)^{-1}(-\cc V+e^{-\g\kappa}\rho v\nabla\n X),U+\nabla\n Du\Big)_{\LL^2(\Om_{-\varepsilon},T)}+(\phi\nabla\n H_0,Du)_{\LL^2(\Om_{-\varepsilon},T)}\nonumber\\
&&=\Big(V,\,\,-\cc(\L_0^*)^{-1}(U+\nabla\n Du)+c(u)(\cc W_0+\eta\varrho_0QX)\Big)_{\LL^2(\Om_{-\varepsilon},T)}\nonumber\\
&&\quad+\Big(v,\,\,e^{-\g\kappa}\rho \<(\L_0^*)^{-1}(U+\nabla\n Du),\nabla\n X\>-e^{-\g\kappa}c(u)\rho\<W_0,\nabla\n X\>\Big)_{\LL^2(\Om_{-\varepsilon})},\nonumber\eeq
which yield (\ref{5.27*}).
\end{proof}

By Theorem \ref{tt4.1} and a similar argument as in the proof of Proposition \ref{p5.1}, we have the following. The details are omitted.
\begin{pro}\label{p5.3} Let $S$ be of $\CC^5.$  Let $\U$ consist of all $(Z,z)\in\LL^2(\Om_{-\varepsilon},T)\times\dot{\WW}^{1,2}(\Om_{-\varepsilon})$ satisfying
\be\begin{cases}\L_0^*Z+\cc Z=\nabla\n Dz,\\
-\Delta_gz=e^{-\g\kappa}\rho\<Z,\nabla\n X\>,\\
\<Z,\nu\>|_{\Ga_{-\varepsilon}}=\<Z,QX\>|_{\Ga_{b_1}}=0,\\
\<Dz,\nu\>|_{\Ga_{-\varepsilon}\cup\Ga_{b_1}}=0,\quad\int_{\Om_{-\varepsilon}}zdg=0.
\end{cases}\label{5.32}\ee
Then $\dim\U<\infty$ and
$$\<Z,\nu\>|_{\Ga_{b_1}}\not=0.$$
\end{pro}

Set
$$\W^0(\Om_{-\varepsilon},T)=\{\,F\in\LL^2(\Om_{-\varepsilon},T)\,|\,\div_gQF\in\LL^2(\Om_{-\varepsilon})\,\},$$ with the norm
$$\|F\|^2_{\W^0(\Om_{-\varepsilon},T)}=\|F\|^2_{\LL^2(\Om_{-\varepsilon},T)}+\|\div_gQF\|^2_{\LL^2(\Om_{-\varepsilon})},$$ and
$$\H^0=\W^0(\Om_{-\varepsilon},T)\times\LL^2(\Om_{-\varepsilon})\times\LL^2(\Ga_{-\varepsilon})\times\LL^2(\Ga_{b_1}).$$
By a similar argument as for (\ref{5.19}), we have the following.
\begin{lem} Let $H_0$ be given in Proposition $\ref{p5.2}.$ For given $(F,f,p,q)\in\H_0,$ problem
\be\begin{cases}
Dv=\nabla\n(V+\b(F,f,p,q) H_0)+QD\varrho_V+F,\\
\div_g(V+\b(F,f,p,q) H_0)=\eta\<V,\nabla\n X\>+f,\\
\int_{\Om_{-\varepsilon}}vdg=0,\quad\<V,QX\>|_{\Ga_{-\varepsilon}}=p,\quad\<V,\nu\>|_{\Ga_{b_1}}=q
\end{cases}\label{5.29}\ee admits a unique solution $(V,v)\in\LL^2(\Om_{-\varepsilon},T)\times\dot{\WW}^{1,2}(\Om_{-\varepsilon}),$ where $\varrho_V$ is given by $(\ref{5.27})$ and
$$\b(F,f,p,q)=-\int_{\Ga_{-\varepsilon}}\<\nabla\n V+QD\varrho_{V}+F,Q\nu\>d\Ga.$$
\end{lem}

Consider problem
\be\left\{\begin{array}{l}Dv=\nabla\n V+F\qfq x\in\Om_{-\varepsilon},\\
\div_gV=\rho v+f\qfq x\in\Om_{-\varepsilon},\\
\<V,QX\>|_{\Ga_{-\varepsilon}}=p,\quad\<V,\nu\>|_{\Ga_{b_1}}=q,\quad
\int_{\Om_{-\varepsilon}}vdg=0,\end{array}\right.\label{n5.33}\ee where $\rho=-\tr_g\Pi.$ Let  $\N_0$ consist of  solutions to
$$\left\{\begin{array}{l}Dv=\nabla\n V\qfq x\in\Om_{-\varepsilon},\\
\div_gV=\rho v\qfq x\in\Om_{-\varepsilon},\\
\<V,QX\>|_{\Ga_{-\varepsilon}}=\<V,\nu\>|_{\Ga_{b_1}}=\int_{\Om_{-\varepsilon}}vdg=0.\end{array}\right.$$ By Theorems \ref{tt4.1} and \ref{t3.1}
$$\N_0\subset\WW^{m+4,2}(\Om_{-\varepsilon},T)\times\WW^{m+5,2}(\Om_{-\varepsilon})$$ if $S$ is of $\CC^{m+5}.$
Set
\be\W^m(\Om_{-\varepsilon},T)=\{\,F\in\WW^{m,2}(\Om_{-\varepsilon},T)\,|\,\div_gQF\in\WW^{m,2}(\Om_{-\varepsilon})\,\},\label{4.16xx}\ee
$$\|F\|^2_{\W^m(\Om_{-\varepsilon},T)}=\|F\|^2_{\WW^{m,2}(\Om_{-\varepsilon},T)}+\|\div_gQF\|^2_{\WW^{m,2}(\Om_{-\varepsilon})}.$$

\begin{thm}\label{t5.1} Let $S$ be of $\CC^{m+5}.$  Let $(F,f,p)\in\W^0(\Om_{-\varepsilon},T)\times\LL^2(\Om_{-\varepsilon})\times\LL^2(\Ga_{-\varepsilon})$ be given. Then there exists
$$q\in\WW^{m+9/2,2}(\Ga_{b_1})$$ such that problem $(\ref{n5.33})$ admits a unique  solution $\N_0^\bot\cap[\LL^2(\Om_{-\varepsilon},T)\times\WW^{1,2}(\Om_{-\varepsilon})].$ Moreover, if
$(F,f,p)\in\W^m(\Om_{-\varepsilon},T)\times\WW^{m,2}(\Om_{-\varepsilon})\times\WW^{m,2}(\Ga_{-\varepsilon}),$ then
\be(V,v,V|_{\Ga_{-\varepsilon}})\in\WW^{m,2}(\Om_{-\varepsilon},T)\times\WW^{m+1,2}(\Om_{-\varepsilon})\times\WW^{m,2}(\Ga_{-\varepsilon},T).\label{n5.35}\ee
\end{thm}

\begin{proof} {\bf Step 1.}\,\,\,Let $(F,f,p,q)\in\H_0$ be given. We solve problem (\ref{5.29}) to have the solution $(V_0,v_0)\in\LL^2(\Om_{-\varepsilon},T)\times\dot{\WW}^{1,2}(\Om_{-\varepsilon}).$ Then
$$\L_0V_0=\G(F,f),\quad \<V_0,QX\>|_{\Ga_{-\varepsilon}}=p,\quad \<V_0,\nu\>|_{\Ga_{b_1}}=q,$$ where  $\G(F,f)=e^{-\g\kappa}[(-\div_gQF)QX+f\nabla\n X].$

We compute $\b(F,f,p,q)$ as follows. Let $W_0$ be given in (\ref{5.23}). Noting that $\<X,\nu\>V_0=\<V_0,QX\>Q\nu+\<V_0,\nu\>X,$ we have, by (\ref{5.23}) and (\ref{2.14}),
\beq\int_{\Ga_{-\varepsilon}}\<\nabla\n V_0,Q\nu\>d\Ga&&=\int_{\Ga_{-\varepsilon}}\frac1{\<X,\nu\>}[p\Pi(Q\nu, Q\nu)+\Pi(X,Q\nu)\<V_0,\nu\>]d\Ga\nonumber\\
&&=\int_{\Ga_{-\varepsilon}}(e^{\g\kappa}\pounds_2p+\pounds_1\<V_0,\nu\>\<W_0,\nu\>)d\Ga\nonumber\\
&&=\Big(p,\pounds_2(e^{\g\kappa}+\<W_0,QX\>)\Big)_{\LL^2(\Ga_{-\varepsilon})}-\Big(q,\pounds_1\<W_0,\nu\>\Big)_{\LL^2(\Ga_{b_1})}\nonumber\\
&&\quad+(W_0,\L_0V_0)_{\LL^2(\Om_{-\varepsilon},T)}-(V_0,\L_0^*W_0)_{\LL^2(\Om_{-\varepsilon},T)},\nonumber\eeq
$$\int_{\Ga_{-\varepsilon}}\<D\varrho_{V_0},\nu\>d\Ga=\int_{\Om_{-\varepsilon}}\div_g\varrho_0D\varrho_{V_0}dg=\int_{\Om_{-\varepsilon}}\varrho_0\eta\<V_0,QX\>dg.$$
Noting that
$$\int_{\Ga_{-\varepsilon}}\<F,Q\nu\>d\Ga=-\int_{\Om_{-\varepsilon}}\div_g\varrho_0QFdg,$$ we obtain
\beq-\b(F,f,p,q)&&=\Big(p,\pounds_2(e^{\g\kappa}+\<W_0,QX\>)\Big)_{\LL^2(\Ga_{-\varepsilon})}-\Big(q,\pounds_1\<W_0,\nu\>\Big)_{\LL^2(\Ga_{b_1})}\nonumber\\
&&\quad+\Big(W_0,\G(F,f)\Big)_{\LL^2(\Om_{-\varepsilon},T)}-(\div_gQF,\varrho_0)_{\LL^2(\Om_{-\varepsilon})}\nonumber\\
&&\quad+(F,QD\varrho_0)_{\LL^2(\Om_{-\varepsilon},T)}.\label{5.34}\eeq

{\bf Step 2.}\,\,\,Consider problem
\be(\id-\A)(V,v)=(V_0, v_0)\qiq \LL^2(\Om_{-\varepsilon},T)\times\dot{\WW}^{1,2}(\Om_{-\varepsilon}),\label{5.31}\ee where $\A$ is defined by (\ref{n5.23}).
By the Fredholm theorem  problem (\ref{5.31}) admits a solution in $\LL^2(\Om_{-\varepsilon},T)\times\dot{\WW}^{1,2}(\Om_{-\varepsilon})$ if and only if
\be\Big((V_0, v_0),(U,u)\Big)_{\LL^2(\Om_{-\varepsilon},T)\times\dot{\WW}^{1,2}(\Om_{-\varepsilon})}=0\qfq (U,u)\in\Z(\id-\A^*).\label{5.36}\ee

Set
\be Z=(\L_0^*)^{-1}(U+\nabla\n Du),\quad z=u\qfq (U,u)\in\Z(\id-\A^*).\label{n5.39}\ee Then $(Z-c(z)W_0,z)$ satisfies (\ref{5.32}) and $(Z-c(z)W_0,z)\in\U.$
Using (\ref{5.29}), (\ref{5.34}) and (\ref{2.14}), we have
\beq&&\Big((V_0,v_0),(U,u)\Big)_{\LL^2(\Om_{-\varepsilon},T)\times\dot{\WW}^{1,2}(\Om_{-\varepsilon})}=\Big(V_0,U\Big)_{\LL^2(\Om_{-\varepsilon},T)}+(Dv_0,Du)_{\LL^2(\Om_{-\varepsilon},T)}\nonumber\\
&&=\Big(V_0,U+\nabla\n Du\Big)_{\LL^2(\Om_{-\varepsilon},T)}+\b(F,f,p,q)c(u)+(F,Du)_{\LL^2(\Om_{-\varepsilon},T)}\nonumber\\
&&=\Big(V_0,\L_0^*Z\Big)_{\LL^2(\Om_{-\varepsilon},T)}+\b(F,f,p,q)c(z)+(F,Dz)_{\LL^2(\Om_{-\varepsilon},T)}\nonumber\\
&&=\Big(\G(F,f),Z-c(z)W_0\Big)_{\LL^2(\Om_{-\varepsilon},T)}+\Big(p,\pounds_2(e^{\g\kappa}c(z)+\<Z-c(z)W_0,QX\>)\Big)_{\LL^2(\Ga_{-\varepsilon})}\nonumber\\
&&\quad-\Big(q,\pounds_1\<Z-c(z)W_0,\nu\>\Big)_{\LL^2(\Ga_{b_1})}+\Big(\div_gQF,c(z)\varrho_0\Big)_{\LL^2(\Om_{-\varepsilon})}+\Big(F,Dz-c(z)QD\varrho_0\Big)_{\LL^2(\Om_{-\varepsilon},T)}\nonumber\\
&&=I(U,u)-\Big(q,\pounds_1\<Z-c(z)W_0,\nu\>\Big)_{\LL^2(\Ga_{b_1})},\label{5.40}\eeq where \beq I(U,u)&&=\Big(\G(F,f),Z-c(z)W_0\Big)_{\LL^2(\Om_{-\varepsilon},T)}+\Big(p,\pounds_2(e^{\g\kappa}c(z)+\<Z-c(z)W_0,QX\>)\Big)_{\LL^2(\Ga_{-\varepsilon})}\nonumber\\
&&\quad+\Big(\div_gQF,c(z)\varrho_0\Big)_{\LL^2(\Om_{-\varepsilon})}+\Big(F,Dz-c(z)QD\varrho_0\Big)_{\LL^2(\Om_{-\varepsilon},T)}.\nonumber\eeq

Set
$$\X=\{\,\<Z-c(z)W_0,\nu\>|_{\Ga_{b_1}}\,|\,Z=(\L_0^*)^{-1}(U+\nabla\n Du),\,\,z=u,\,\,(U,u)\in\Z(\id-\A^*)\,\}.$$ By Proposition \ref{p5.3}
$$\dim\Z(\id-\A^*)=\dim\X.$$
Let $\{\,\<Z_i-c(z_i)W_0,\nu\>\,\}_{i=1}^k$ be a basis of $\X$ such that
$$\Big(\<Z_i-c(z_i)W_0,\nu\>,\pounds_1\<Z_j-c(z_j)W_0,\nu\>\Big)_{\Ga_{b_1}}=\delta_{ij}\qfq1\leq i,\,j\leq k.$$
For given $(F,f,p)\in\W^0(\Om_{-\varepsilon},T)\times\LL^2(\Om_{-\varepsilon})\times\LL^2(\Ga_{-\varepsilon}),$ let
$$q=\sum_{i=1}^kc_i\<Z_i-c(z_i)W_0,\nu\>|_{\Ga_{b_1}},$$ where
$$c_i=\Big((V_0,v_0),(U_i,u_i)\Big)_{\LL^2(\Om_{-\varepsilon},T)\times\dot{\WW}^{1,2}(\Om_{-\varepsilon})}-I(U_i,u_i),$$ and $(U_i,z_i)\in\Z(\id-\A^*)$ such that
$$Z_i=(\L_0^*)^{-1}(U_i+\nabla\n Du_i),\,\,z_i=u_i\qfq1\leq i\leq k.$$ Then (\ref{5.36}) hold true. Thus problem (\ref{5.31}) admits a unique solution $(V,v)\in\LL^2(\Om_{-\varepsilon},T)\times\dot{\WW}^{1,2}(\Om_{-\varepsilon}),$ which implies that
$$(V+\b_0H_0,v)\in \LL^2(\Om_{-\varepsilon},T)\times\dot{\WW}^{1,2}(\Om_{-\varepsilon})$$ is a solution to problem (\ref{n5.33}) with
$$q=q_0+\b_0\<H_0,\nu\>\in\WW^{m+9/2,2}(\Ga_{b_1}),$$ where
$$\b_0=-\int_{\Ga_{-\varepsilon}}\<\nabla\n V+F,Q\nu\>d\Ga.$$

Finally, (\ref{n5.35}) follows from Theorems \ref{tt4.1} and \ref{t3.1}.
\end{proof}

Let $W_0$ solve problem (\ref{5.23}). We define
$$\aleph=\{\,(\pounds_2(e^{\g\kappa}c(z)+\<Z-c(z)W_0,QX\>),\,\pounds_1\<Z-c(z)W_0,\nu\>)\,|$$
\be(Z,z)\,\mbox{is given in (\ref{n5.39})}\,\}.\label{n5.41x}\ee By Theorem Theorems \ref{tt4.1} and \ref{t3.1}
$$\aleph\subset\WW^{m+9/2,2}(\Ga_{-\varepsilon})\times\WW^{m+9/2,2}(\Ga_{b_1})$$ if $S$ is of $\CC^{m+5}.$

Consider problem
\be\left\{\begin{array}{l}Dv=\nabla\n V\qfq x\in\Om_{-\varepsilon},\\
\div_gV=\rho v\qfq x\in\Om_{-\varepsilon},\\
\<V,QX\>|_{\Ga_{-\varepsilon}}=p,\quad\<V,\nu\>|_{\Ga_{b_1}}=q.\end{array}\right.\label{5.41}\ee Let  $\N_1$ consist of  solutions to
\be\left\{\begin{array}{l}Dv=\nabla\n V\qfq x\in\Om_{-\varepsilon},\\
\div_gV=\rho v\qfq x\in\Om_{-\varepsilon},\\
\<V,QX\>|_{\Ga_{-\varepsilon}}=\<V,\nu\>|_{\Ga_{b_1}}=0.\end{array}\right.\label{5.42x}\ee Then
$$\N_1\subset\WW^{m+4,2}(\Om_{-\varepsilon},T)\times\WW^{m+5,2}(\Om_{-\varepsilon})$$ if $S$ is of $\CC^{m+5}.$

\begin{thm}\label{t5.2} Let $S$ be of $\CC^{m+5}.$ Problem $(\ref{5.41})$ admits a unique solution
$$(V,v)\bot\N_1\qiq \LL^2(\Om_{-\varepsilon},T)\times\WW^{1,2}(\Om_{-\varepsilon})$$ if and only if
\be(p,q)\bot\aleph\qiq\LL^2(\Ga_{-\varepsilon})\times\LL^2(\Ga_{b_1}).\label{5.44}\ee Moreover, if $(p,q)\in\WW^{j,2}(\Ga_{-\varepsilon})\times\WW^{j,2}(\Ga_{b_1})$ with $(\ref{5.44}),$ then
$$(V,v)\in\WW^{j,2}(\Om_{-\varepsilon},T)\times\WW^{j+1,2}(\Om_{-\varepsilon})\qfq 0\leq j\leq m+4.$$
\end{thm}

\begin{proof} Let $(V,v)$ be a solution to problem (\ref{5.41}). Then $(V, v-\theta)$ solves problem (\ref{n5.33}) with
\be(F,f,p,q)=(0,\theta\rho,p,q),\label{5.42}\ee where $\theta=(v,1)_{\LL^2(\Om_{-\varepsilon})}/m(\Om_{-\varepsilon}).$ Let ${\cal P}$ be the orthogonal projection operator from $\LL^2(\Ga_{-\varepsilon})\times\LL^2(\Ga_{b_1})$ on to $\aleph.$ By (\ref{5.40}) problem (\ref{n5.33}) admits a solution if and only
\beq&&\Big(p,\pounds_2(e^{\g\kappa}c(z)+\<Z-c(z)W_0,QX\>)\Big)_{\LL^2(\Ga_{-\varepsilon})}+\Big(q,\pounds_1\<Z-c(z)W_0,\nu\>\Big)_{\LL^2(\Ga_{b_1})}\nonumber\\
&&=-\theta\int_{\Om_{-\varepsilon}}\rho e^{-\g\kappa}\<Z-c(z)W_0,\nabla\n X\>dg=\theta\int_{\Om_{-\varepsilon}}\Delta_gzdg=0,\nonumber\eeq where $Z-c(z)W_0$ satisfies  (\ref{5.32}).
Thus the desired results follow.
\end{proof}

\setcounter{equation}{0}
\def\theequation{6.\arabic{equation}}
\section{Proofs of the Main Results }
\hskip\parindent {\bf Proof of Theorem \ref{t1.1}.} Let
$$U\in T^2_{\sym}S.$$ Consider problem
\be\left\{\begin{array}{l}Dv=\nabla\n V+F\qfq x\in S,\\
\div_gV+v\tr_g\Pi=f\qfq x\in S,\end{array}\right.\label{6.1}\ee
where $(V,v)\in\LL^2(S,T)\times\WW^{1,2}(S)$ is the unknown and
\be F=Q[D(\tr_g U)-\div_gU],\quad f=-\tr_gU(Q\nabla\n\cdot,\cdot)\qfq x\in S.\label{6.2}\ee

For $y\in\WW^{1,2}(S,\R^3),$ let
\be 2v=\nabla y(e_2,e_1)-\nabla y(e_1,e_2)\qfq x\in S,\label{6.3}\ee
\be V=(\nabla\n)^{-1}(Dv-F)\qfq x\in S,\label{6.4}\ee where $e_1,$ $e_2$ is an orthonormal basis of $T_xS$ with positive orientation.

 By \cite[Section 2]{Yao2017}, there is a $y\in\WW^{1,2}(S,\R^3)$ to solve problem (\ref{01}) if and only if $(V,v),$ being given in (\ref{6.4}) and (\ref{6.3}), respectively, solves problem (\ref{6.1}). In that case, we have
\be\left\{\begin{array}{l}\nabla_{e_1}y=U(e_1,e_1)e_1+[v+U(e_1,e_2)]e_2-\<QV,e_1\>\n,\\
\nabla_{e_2}y=[-v+U(e_1,e_2)]e_1+U(e_2,e_2)e_2-\<QV,e_2\>\n,\end{array}\right.\qfq x\in S.\label{y6.6}\ee
 Moreover, from \cite[Theorem 2.1]{Yao2017},   $(V,v)$ is a solution to problem (\ref{6.1}) if and only if $v$ solves problem
 \be\<D^2v,Q^*\Pi\>+\frac1\kappa X_0v+v\kappa\tr_g\Pi=\kappa f+\frac1\kappa\<X_0,F\>+\<DF,Q^*\Pi\>\qfq x\in S,\quad\kappa\not=0\label{6.6}\ee where $f$ and $F$ are given in (\ref{6.2}).\\

Let $U\in\WW^{m+1,2}(S,T^2_{\sym})$ be given.  We will find a solution
$$(V,v)\in\WW^{m-1,2}(S,T)\times\WW^{m,2}(S)$$ to problem (\ref{6.1}) as follows.

Consider the region $\Om_{-\varepsilon}$ in Theorem \ref{t5.1}. By (\ref{6.2})
$$(F,f,0)\in\W^{m-1}(\Om_{-\varepsilon},T)\times\WW^{m-1,2}(\Om_{-\varepsilon})\times\WW^{m-1,2}(\Ga_{-\varepsilon})$$ By Theorem \ref{t5.1} there exists  $q\in\WW^{m+9/2,2}(\Ga_{b_1})$ such that problem
\be\left\{\begin{array}{l}Dv=\nabla\n V+F\qfq x\in\Om_{-\varepsilon},\\
\div_gV=\rho v+f\qfq x\in\Om_{-\varepsilon},\\
\<V,QX\>|_{\Ga_{-\varepsilon}}=p,\quad\<V,\nu\>|_{\Ga_{b_1}}=q,\quad
\int_{\Om_{-\varepsilon}}vdg=0,\end{array}\right.\label{6.7}\ee where $\rho=-\tr_g\Pi,$
 admits a solution $(U,u)$ satisfying
$$(U,u,U|_{\Ga_{-\varepsilon}})\in\WW^{m-1,2}(\Om_{-\varepsilon},T)\times\WW^{m,2}(\Om_{-\varepsilon})\times\WW^{m-1,2}(\Ga_{-\varepsilon},T).$$

The above solution can be extended to the region $S_1=S/\Om_{-\varepsilon}$ below. Since $S_1$ is a noncharacteristic region, the boundary operators $\T_1$ and $T_2$ as in \cite{Yao2017} can defined.
 Let $x\in\Ga_{-\varepsilon}$ be given. $\mu\in T_xS_1$ with $|\mu|=1$ is said to be the {\it noncharacteristic normal} outside $S_1$ if there is a curve
$\zeta:$ $(0,\iota)\rw S$ such that
$$\zeta(0)=x,\quad \zeta'(0)=-\mu,\quad\Pi(\mu,Y)=0\qfq Y\in T_x\Ga_{-\varepsilon}.$$ Let $\mu$ be the the noncharacteristic normal field along $\Ga_{-\varepsilon}.$
 We define boundary operators $\T_i:$ $T_xM\rw T_xM$ by
\be \T_iX=\frac{1}{2}\Big[Y+(-1)^i\chi(\mu,Y)\omega(Y)Q\nabla\n Y]\qfq Y\in T_xM,\quad i=1,\,\,2,\label{xn4.14}\ee where
\be\chi(\mu,Y)=\sign\det\Big(\mu,Y,\n\Big),\quad \omega(Y)=\frac{1}{\sqrt{-\kappa}}\sign\Pi(Y,Y),\label{rho4.3}\ee and  $\sign$ is the sign function. By a similar argument for \cite[Theorem 4.2]{Yao2017}
problem
$$\begin{cases}\<D^2v,Q^*\Pi\>+\frac1\kappa X_0v+v\kappa\tr_g\Pi=\kappa f+\frac1\kappa\<X_0,F\>+\<DF,Q^*\Pi\>\qfq x\in S_1,\\
w=u,\quad \frac1{\sqrt{2}}\<Dw,(\T_1-\T_2)\a_t\>=\frac1{\sqrt{2}}\<Du,(\T_1-\T_2)\a_t\>\qfq x\in\Ga_{-\varepsilon}
\end{cases}$$
admits a unique solution $w\in\WW^{m+1,2}(S_1),$ where $u$ is the component of the solution $(U,u)$ to problem (\ref{6.7}). Let
$$W=(\nabla\n)^{-1}(Dw-F)\qfq x\in S_1.$$
We define
$$(V,v)=\begin{cases}(U,u)\qfq x\in\Om_{-\varepsilon},\\
(W,w)\qfq x\in S_1.\end{cases}$$ Then $(V,v)\in\WW^{m-1,2}(S,T)\times\WW^{m,2}(S)$ is a solution to (\ref{6.1}). Then by (\ref{y6.6}) $y\in\WW^{m,2}(S,\R^3).$ Thus by \cite[Lemma 4.3]{Yaobook}
$$\sym DW=U-\<y,\n\>\Pi\in\WW^{m,2}(S,T).$$ The proof is complete.\hfill$\Box$\\

{\bf Proof of Theorem 1.2}\,\,\,Let $y\in\WW^{2,2}(S,\R^3)$ be an  infinitesimal isometry. Let
$$ 2v=\nabla y(e_2,e_1)-\nabla y(e_1,e_2),\quad V=(\nabla\n)^{-1}Dv\qfq x\in S,$$ where $e_1,$ $e_2$ is an orthonormal basis of $T_xS$ with positive orientation. Then $(V,v)$ solves problem
\be\left\{\begin{array}{l}Dv=\nabla\n V\qfq x\in S,\\
\div_gV+v\tr_g\Pi=0\qfq x\in S.\end{array}\right.\label{6.10}\ee By (\ref{y6.6}) with $U=0,$ $y\in\WW^{2,2}(S,\R^3)$ implies $(V,v)\in\WW^{1,2}(S,T)\times \WW^{2,2}(S).$

Consider the region $\Om_{-\varepsilon}$ as in Theorem \ref{t5.2}.
By \cite[Theorem 9.19]{GNT} $v|_{\Ga_{b_1}}\in\CC^{m+4,\a}(\Ga_{b_1})$ due to that $v$ solves an elliptic equation on $S_+.$ In addition, since $v$ solves a hyperbolic equation on $S_-,$
 $v|_{\Ga_{-\varepsilon}}\in\WW^{2,2}(\Ga_{-\varepsilon})$ by \cite{Yao2017}. Thus
 $$V|_{\Ga_{b_1}}\in\CC^{m+1}(\Ga_{b_1}),\quad V|_{\Ga_{-\varepsilon}}\in\WW^{1,2}(\Ga_{-\varepsilon}).$$
Set
$$p=\<V,QX\>|_{\Ga_{-\varepsilon}},\quad q=\<V,\nu\>|_{\Ga_{b_1}}.$$ We decompose as the direct sum
$$(V,v)=(V_1,v_1)+(V_0,v_0)\qiq\LL^2(\Om_{-\varepsilon},T)\times\WW^{1,2}(\Om_{-\varepsilon}),$$ where $(V_0,v_0)\in\N_1$ and $\N_1$ is given by (\ref{5.42x}). Then $$(V_0,v_0)\in\WW^{m+4,2}(\Om_{-\varepsilon},T)\times\WW^{m+5,2}(\Om_{-\varepsilon}),$$ and $(V_1,v_1)\in\LL^2(\Om_{-\varepsilon},T)\times\WW^{1,2}(\Om_{-\varepsilon})$ solves problem (\ref{5.41}).
By Theorem \ref{t5.2}
$$(p,q)\bot\aleph\qiq\LL^2(\Ga_{-\varepsilon})\times\LL^2(\Ga_{b_1}),$$ and
$$\|(V_1,v_1)\|_{\WW^{1,2}(\Om_{-\varepsilon},T)\times\WW^{2,2}(\Om_{-\varepsilon})}\leq C\|(p,q)\|^2_{\WW^{1,2}(\Ga_{-\varepsilon})\times\WW^{1,2}(\Ga_{b_1})},$$ where $\aleph$ is given in (\ref{n5.41x}).

Next, for given $\iota>0$ arbitrarily, we take $(p_\iota,q_\iota)\in\WW^{m+1,2}(\Ga_{-\varepsilon})\times\WW^{m+1,2}(\Ga_{b_1})$ such that
$$\|(p,q)-(p_\iota,q_\iota)\|_{\WW^{1,2}(\Ga_{-\varepsilon})\times\WW^{1,2}(\Ga_{b_1})}\leq\iota.$$
Denote by $(\hat p_\iota,\hat q_\iota)$ by the orthogonal projection of $(p_\iota,q_\iota)$ from $\LL^2(\Ga_{-\varepsilon})\times\LL^2(\Ga_{b_1})$ to $\aleph^\bot$ in $\LL^2(\Ga_{-\varepsilon})\times\LL^2(\Ga_{b_1}).$ Since $\dim\aleph<\infty,$ we have
$$\|(p,q)-(\hat p_\iota,\hat q_\iota)\|_{\WW^{1,2}(\Ga_{-\varepsilon})\times\WW^{1,2}(\Ga_{b_1})}\leq C\iota.$$
By Theorem \ref{t5.2} problem (\ref{5.41}) admits a unique solution $(V_\iota,v_\iota)\in\N_1^\bot$ in $\LL^2(\Om_{-\varepsilon})\times\WW^{1,2}(\Om_{-\varepsilon}),$ and
$$\|(V_\iota,v_\iota)\|_{\WW^{m+1,2}(\Om_{-\varepsilon},T)\times\WW^{m+2,2}(\Om_{-\varepsilon})}\leq C\|(\hat p_\iota,\hat q_\iota)\|_{\WW^{m+1,2}(\Ga_{-\varepsilon})\times\WW^{m+1,2}(\Ga_{b_1})}.$$
Thus
$$(V_\iota,v_\iota)\in\CC^{m-1}_B(\Om_{-\varepsilon}),$$ and
$$\|(V_1,v_1)-(V_\iota,v_\iota)\|_{\WW^{1,2}(\Ga_{-\varepsilon})\times\WW^{1,2}(\Ga_{b_1})}\leq C\iota.$$

Now we extend the domains of $(V_1,v_1),$  $(V_0,v_0),$ and $(V_\iota,v_\iota)$ from $\Om_{-\varepsilon}$ to $S,$ respectively, as in the proof of Theorem \ref{t1.1}, with the same notations such that
$$(V,v)=(V_1,v_1)+(V_0,v_0)\qiq\WW^{1,2}(S,T)\times\WW^{2,2}(S),$$
$$(V_\iota,v_\iota)\in\CC^{m-1}_B(S),$$ $(V_1,v_1),$  $(V_0,v_0),$ and $(V_\iota,v_\iota)$ satisfy problem (\ref{6.10}) on the region $S,$ respectively, and
$$\|(V,v)-(V_\iota+V_0,v_\iota+v_0)\|_{\WW^{1,2}(S,T)\times\WW^{2,2}(S)}\leq C\iota.$$ The proof is complete by the formula (\ref{y6.6}). \hfill$\Box$\\

{\bf Proof of Theorem \ref{t1.3}}\,\,\,As in $\cite{HoLePa}$ we conduct in $2\leq i\leq m.$  Let
$$y_\varepsilon=\sum_{j=0}^{i-1}\varepsilon^jz_j$$ be an $(i-1)$th order isometry of class $\CC^{2+4(m-i+1)}_B(S,\R^3),$ where $z_0=\id$ and $z_2=y$ for some $i\geq2.$ Then
$$\sum_{j=0}^k\na^Tz_j\na z_{k-j}=0\qfq 1\leq k\leq i-1.$$

Next, we shall find out $z_i\in\CC^{2+4(m-i)}_B(S,\R^3)$ such that
$$\phi_\varepsilon=y_\varepsilon+\varepsilon^iz_i$$ is an $i$th order isometry.
By Corollary  $\ref{c1.1}$ there exists  a solution $z_i\in\CC^{2+4(m-i)}_B(S,\R^3)$ to problem
$$\sym \nabla z_i=-\frac{1}{2}\sym\sum_{j=1}^{i-1}\na^Tz_j\na z_{i-j}$$ which satisfies
\beq\|z_i\|_{\CC^{2+4(m-i)}_B(S,\R^3)}&&\leq C\|\sum_{j=1}^{i-1}\sym\na^Tz_j\na z_{i-j}\|_{\CC^{2+4(m-i)+3}_B(S,\R^3)}\nonumber\\
&&\leq C\sum_{j=1}^{i-1}\|z_j\|_{\CC^{2+4(m-i+1)}_B(S,\R^3)}\|z_{i-j}\|_{\CC^{2+4(m-i+1)}_B(S,\R^3)}.\nonumber\eeq
The conduction completes. \hfill$\Box$\\

\appendix{\bf Appendix: A theorem of  Lax-Milgram's type}\\

The following theorem is an improved version of the Lax-Milgram theorem.

{\bf Theorem A}\,\,{\it  Let $(\V,\<\cdot,\cdot\>_\V)$ be a Hilbert space. Suppose that $\W\subset\V$ is a linear subspace of $\V$ such that $(\W,\<\cdot,\cdot\>_\W)$ is an inner product space, which may not be complete. Let $\b:$ $\V\times\W\rw\R$ be a bilinear functional satisfying
\be|\b(v,w)|\leq C\|v\|_\V\|w\|_\W\qfq v\in\V,\,\,w\in\W,\label{6.11}\ee
\be\b(w,w)\geq\si\|w\|^2_\V\qfq w\in\W.\ee Then there is a closed linear subspace $\V'\subset\V$ such that for  given bounded linear functional $\F$ on $\V,$ there  exists a unique $v\in\V'$ satisfying
$$\F(w)=\b(v,w)\qfq w\in\W,$$
\be\|v\|\leq\frac1\si\|\F\|_{\V^*}.\label{6.14}\ee}

\begin{proof} For given $w\in\W,$ from (\ref{6.11}) $\b(\cdot,w)$ is a bounded linear functional on $\V.$ The Riesz representation theorem implies that there exists a bounded linear operator $A:$ $\W\rw\V$ such that
$$\b(v,w)=\<v,A w\>_\V\qfq v\in\V,\,\,w\in\W.$$
Then $\|A\|\leq C.$ Thus
$$\<w,A w\>_\V=\b(w,w)\geq\si\|w\|^2_\V\qfq w\in\W,$$ which yield
\be\|A w\|_\V\geq\si\|w\|_\V\qfq w\in\W.\label{6.13}\ee

Let $\V'$ be the closure of $R(A)$ in $\V.$ Then for given $y\in\V'$ there is a sequence $\{w_k\}_1^\infty\subset\W$ such that
$$\|y-Aw_k\|_\V\rw0\quad\mbox{as}\quad k\rw\infty.$$ Define
$$\varphi(y)=\lim_{k\rw\infty}\F(w_k).$$ By (\ref{6.13}) $\var$ is a bounded linear functional on $\V'.$ By the Riesz representation theorem there exists a unique $v\in\V'$ such that
$$\var(y)=\<v,y\>_\V\qfq y\in\V'.$$ In particular, letting $y=Aw$ for $w\in\W$ yields
$$\F(w)=\var(y)=\<v,Aw\>_\V=\b(v,w)\qfq w\in\W.$$ Finally, (\ref{6.14}) follows from (\ref{6.13}).
\end{proof}

{\bf Compliance with Ethical Standards}

Conflict of Interest: The author declares that there is no conflict of interest.

Ethical approval: This article does not contain any studies with human participants or animals performed by the authors.

 \end{document}